\documentclass[11pt,a4paper]{article}
\usepackage[OT1]{fontenc}
\usepackage{amsthm,amsmath,amssymb,amsfonts,color,enumerate}
\usepackage[authoryear]{natbib}

\usepackage[colorlinks, allcolors=blue]{hyperref}
\newcommand{\eqrefb}[1]{(\ref{#1})}
\newcommand{\refb}[1]{\ref{#1}}
\usepackage{multibib}
\newcites{suppl}{References}

\usepackage{graphicx}
\usepackage{array}
\usepackage{float}
\usepackage{setspace}
\usepackage{titling}
\usepackage{multirow}
\usepackage{algorithm2e}
\usepackage{booktabs}
\numberwithin{equation}{section}

\newtheoremstyle{dotless}{}{}{\itshape}{}{\bfseries}{}{ }{}
\theoremstyle{dotless}
\newtheorem{theorem}{Theorem}[section]
\newtheorem{corollary}[theorem]{Corollary}
\newtheorem{assump}[theorem]{Assumption}
\newtheorem{lemma}[theorem]{Lemma}

\newtheoremstyle{definition}{}{}{}{}{\bfseries}{}{ }{}
\theoremstyle{definition}

\newtheorem{remark}[theorem]{Remark}

\newcommand{\Pb}{\mathbb{P}}
\newcommand{\E}{\mathbb{E}}

\newcommand{\R}{\mathbb{R}}

\newcommand{\Z}{\mathbb{Z}}
\newcommand{\N}{\mathbb{N}}

\newcommand{\IF}{\mathcal{IF}}
\newcommand{\floor}[1]{\lfloor #1 \rfloor}

\newcommand{\convd}{\overset{\mathcal{D}}{\Longrightarrow}}
\newcommand{\eqd}{\overset{\mathcal{D}}{=}}
\newcommand{\convp}{\overset{\mathbb{P}}{\Longrightarrow}}

\newcommand{\hatSigma}{\hat{\Sigma}_m}
\newcommand{\hatGamma}{\hat{\Gamma}}
\newcommand{\hattheta}{\hat{\theta}}
\newcommand{\hatbeta}{\hat{\beta}}

\newcommand{\hatD}{\hat{D}}
\newcommand{\hatE}{\hat{E}}
\newcommand{\hatP}{\hat{P}}
\newcommand{\hatQ}{\hat{Q}}

\newcommand{\op}{o_{\mathbb{P}}}
\newcommand{\Op}{\mathcal{O}_{\mathbb{P}}}
\newcommand{\hatFij}{\hat{F}_i^j}
\newcommand{\supkinf}{\sup_{k=1}^\infty}

\newcommand{\maxijk}{\max_{\substack{i,j=1\\i<j}}^k}

\newcommand{\WmSig}{W_{m,1}^{\Sigma}}
\newcommand{\WmSigII}{W_{m,2}^{\Sigma}}
\newcommand{\convm}{\underset{m \to \infty}{\Longrightarrow}}
\newcommand{\tildew}{\tilde{w}}

\DeclareMathOperator{\Cov}{Cov}

\newcommand{\revTwo}[1]{{#1}}

\begin{document}
\onehalfspacing
\title{{A new approach for open-end sequential change point monitoring}}

{
\author{
  {\small Josua G\"osmann}\\
  {\small Ruhr-Universit\"at Bochum }\\
  {\small Fakult\"at f\"ur Mathematik}\\
  {\small 44780 Bochum, Germany} \\
  {\small \href{mailto:josua.goesmann@ruhr-uni-bochum.de}{josua.goesmann@ruhr-uni-bochum.de}}\\
  {\small (corresponding author)}
  \and
  {\small Tobias Kley}\\
  {\small University of Bristol}\\
  {\small School of Mathematics}\\
  {\small Bristol BS8 1UG, United Kingdom}\\
  {\small \href{mailto:tobias.kley@bristol.ac.uk}{tobias.kley@bristol.ac.uk}}
\and
  {\small Holger Dette}\\
  {\small Ruhr-Universit\"at Bochum }\\
  {\small Fakult\"at f\"ur Mathematik}\\
  {\small 44780 Bochum, Germany} \\
  {\small \href{mailto:holger.dette@ruhr-uni-bochum.de}{holger.dette@ruhr-uni-bochum.de}}
}

\maketitle

\begin{abstract}
We propose a new sequential monitoring scheme for changes in the parameters of a multivariate time series.
In contrast to procedures proposed in the literature which compare an estimator from the training sample with an estimator calculated from the remaining data, we suggest to divide the sample at each time point after the training sample.
Estimators from the sample before and after all separation points are then continuously compared calculating a maximum of norms of their differences.
For open-end scenarios our approach yields an asymptotic level $\alpha$ procedure, which is consistent under the alternative of a change in the parameter.
By means of a simulation study it is demonstrated that the new method outperforms the commonly used procedures with respect to power and the feasibility of our approach is illustrated by analyzing two data examples.
\end{abstract}

MSC classification: 62L99, 62F03

JEL classification: C01,C22\\

Keywords and phrases: change point analysis, open-end procedures, sequential monitoring

\section{Introduction}
\label{sec1}

Nowadays, nearly all fields of applications require sophisticated statistical modeling and statistical inference to draw scientific conclusions from the observed data.
In many cases data is time dependent and the involved model parameters or the model itself may not be necessarily stable. In such situations it is of particular importance to detect changes in the processed data as soon as possible and to adapt the statistical analysis accordingly.
These changes are usually called {\it change points} or {\it structural breaks}.
Due to its universality, methods for change point analysis have a vast field of possible applications - ranging from natural sciences, like biology and meteorology, to humanities, like economics, finance, and social sciences.
Since the seminal papers of \cite{Page1954, Page1955} the problem of detecting change points in time series has received substantial attention in the statistical literature.
The contributions to this field can be roughly divided into the areas of {\it retrospective} and {\it sequential} change point analysis.

In the retrospective case, historical data sets are examined with the aim to test for changes and identify their position within the data.
In this setup, the data is assumed to be completely available before the statistical analysis is started (a-posteriori analysis). 
A comprehensive overview of retrospective change point analysis can be found in \cite{Aue2013}.
In many practical applications, however, data arrives consecutively and breaks can occur at any new data point.
In such cases the statistical analysis for changes in the processed data has to start immediately with the target to detect changes as soon as possible.
This field of statistics is called \textit{sequential change point detection} or \textit{online change point detection}.

In the major part of the 20th century the problem of sequential change point detection was tackled using procedures, which are optimized to have a minimal detection delay but do usually not control the probability of a false alarm (type I error).
These methods are called control charts and a comprehensive review can be found in \cite{Lai1995,Lai2001}.
A new paradigm was then introduced by \cite{Chu1996}, who use initial data sets and therefrom employ invariance principles to also control the type I error.
The methods developed under this paradigm [see below] can again be subdivided into {\it closed-end} and {\it open-end}  approaches.
In closed-end scenarios monitoring is stopped at a fixed pre-defined point of time, while in open-end scenarios monitoring can - in principle - continue forever if no change point is detected.

In the paper at hand we develop a new approach for sequential change point detection in an open-end scenario.
To be more precise let $\{X_t\}_{t \in \Z}$ denote a $d$-dimensional time series and let $F_t$ be the distribution function of the random variable $X_t$ at time $t$.
We are studying monitoring procedures for detecting changes of a parameter $\theta_t = \theta(F_t)$, where $\theta= \theta(F)$ is a $p$-dimensional parameter of a distribution function $F$ on $\R^{d}$ (such as the mean, variance, correlation, etc.).
In particular we will develop a decision rule for the hypothesis of a constant parameter, that is
\begin{align}\label{globalhypo0}
H_0&:\; \theta_1 = \dots = \theta_m = \theta_{m+1} =\theta_{m+2} = \ldots ~,
\end{align}
against the alternative that the parameter changes (once) at some time $ m+k^{\star}$ with $k^\star \geq 1$, that is
\begin{align}\label{globalhypo1}
H_1&:\; \exists k^{\star} \in \N:\;\;
\theta_1 = \dots =\theta_{m+k^{\star}-1} \neq \theta_{m+k^{\star}} = \theta_{m+k^{\star}+1} = \ldots .
\end{align}
In this setup, which was originally introduced by \cite{Chu1996}, the first $m$ observations are assumed to be stable and will serve as an initial training set.
The problem of sequential change point detection in the hypotheses paradigm as pictured above has received substantial interest.
Since the seminal paper of \cite{Chu1996} several authors have worked in this area.
\cite{Berkes2004} designed a detector for changes in the coefficient in the parameters of a GARCH-process.
\cite{Horvath2004}, \cite{Aue2006}, \cite{Aue2009}, \cite{Fremdt2015} and \cite{Aue2014} developed methodology for detecting changes in the coefficients of a linear model, while \cite{Wied2013} and \cite{Pape2016} considered sequential monitoring schemes for changes in special functionals such as the correlation or variance.
A MOSUM-approach was employed by \cite{Leisch2000}, \cite{Horvath2008} or \cite{ChenTian2010} to monitor the mean and linear models, respectively.
Recently, \cite{Hoga2017} used an $\ell_{1}$-norm to detect changes in the mean and variance of a multivariate time series, \cite{Kirch2018} defined a unifying framework for detecting changes in different parameters with the help of several statistics and \cite{Otto2020} considered a Backward CUSUM, which monitors changes based on recursive residuals in a linear model.
A helpful but not exhaustive overview of different sequential procedures can be found in Section 1, in particular Table 1, of \cite{Anatolyev2018}.

A common feature of all procedures in the cited literature consists in the comparison of estimators from different subsamples of the data.
To be precise, let $X_{1},\ldots, X_{m}$ denote an initial training sample and $X_{1}, \ldots , X_{m}, \ldots , X_{m+k}$
be the available data at time $m+k$.
Several authors propose to investigate the differences
\begin{align}
&\hattheta_{1}^{m} - \hattheta_{m+1}^{m+k}~,\label{def:ordCusum}
\end{align}
(in dependence of $k$), where $\hattheta_{i}^{j} $ denotes the estimator of the parameter from the sample $X_{i}, \ldots  , X_{j}$.
In the sequential change point literature monitoring schemes based on the differences \eqref{def:ordCusum} are 
usually called (ordinary) CUSUM procedures and have been considered by \cite{Horvath2004}, \cite{Aue2006, Aue2009, Aue2014}, \cite{Schmitz2010} or \cite{Hoga2017}.
Other authors suggest using a function of the differences
\begin{align}\label{def:pageCusum}
& \big\{ \hattheta_{1}^{m} - \hattheta_{m+j+1}^{m+k} \big\}_{j=0,\ldots,k-1}~
\end{align}
(in dependence of $k$) and the corresponding procedures are usually called Page-CUSUM tests [see \cite{Fremdt2015}, \cite{Aue2015}, or \cite{Kirch2018} among others].
As an alternative we propose - following ideas of \cite{Dette2019} - a monitoring scheme based on a function of the differences
\begin{align}\label{def:retroCUSUM}
\big\{ \hattheta_{1}^{m+j} - \hattheta_{m+j+1}^{m+k} \big\}_{j=0,\ldots,k-1} \;.
\end{align}
A possible advantage of \eqref{def:retroCUSUM} over \eqref{def:ordCusum} is the screening for all potential positions of the change point, which takes into account that the change point not necessarily comes with observation $X_{m+1}$ and so $\hattheta_{m+1}^{m+k}$ maybe `corrupted' by pre-change observations.
This issue is also partially addressed by \eqref{def:pageCusum}, where different positions are examined and compared with the estimator of the parameter from the training sample.  
We will demonstrate in Section \ref{sec4} that sequential monitoring schemes based on the differences \eqref{def:retroCUSUM} yield a substantial improvement in power compared to the commonly used methods based on \eqref{def:ordCusum} and \eqref{def:pageCusum}.
To avoid misunderstandings, the reader should note that a (total) comparison based on differences of the form \eqref{def:retroCUSUM}, is typically also called a CUSUM-approach in the retrospective change point analysis, see \cite{Aue2013} for a comprehensive overview of (retrospective) change point analysis.

The present paper is devoted to a rigorous statistical analysis of a sequential monitoring based on the differences defined in \eqref{def:retroCUSUM} in the context of an open-end scenario.
In Section~\ref{sec2} we introduce the new procedure and develop a corresponding asymptotic theory to obtain critical values such that monitoring can be performed at a controlled type I error.
The theory is broadly applicable to detect changes in a general parameter $\theta$ of a multivariate time series.
As all monitoring schemes in this context the method depends on a weight function and we also discuss the choice of this function.
In particular we establish an interesting result regarding this choice and establish a connection to corresponding ideas made by \cite{Horvath2004} and \cite{Fremdt2015}, which may also be of interest in closed-end scenarios.

In Section \ref{sec3} we discuss several special cases and demonstrate that the new methodology is applicable to detect changes in the mean and the parameters of a linear model.
We present a small simulation study in Section \ref{sec4}, where we compare our approach to those developed by \cite{Horvath2004} and \cite{Fremdt2015}.
In particular we demonstrate that the monitoring scheme based on the differences \eqref{def:retroCUSUM} yields a test with a controlled type I error and a smaller type II error than the procedures in the cited references.
In Section~\ref{sec:RealData} we illustrate our approach and compare it to other monitoring schemes by applying it to two examples where the parameter of a linear model of financial data is monitored around the time of the United Kingdom European Union membership referendum 2016.
Finally, all proofs are deferred to the online appendix [see \cite{Goesmann2020appendix}], in which we additionally provide some extra simulation results and briefly discuss how our statistic can be used in closed-end scenarios.

\section{Asymptotic properties}
\label{sec2}
Throughout this paper let $F$ denote a $d$-dimensional distribution function and $\theta =\theta(F)$ a $p$-dimensional
parameter of $F$.
We will denote by 
\begin{align}\label{eq:edf}
\hatFij(z) = \dfrac{1}{j-i+1}\sum_{t=i}^j I\{X_t \leq z \}~
\end{align}
the empirical distribution function of observations $X_i,\dots,X_j$ (here the inequality is understood component-wise) and consider the canonical estimator $\hattheta_i^j = \theta(\hatFij)$ of the parameter $\theta$ from the sample $X_{i}, \ldots ,X_{j}$.
 
To test the hypotheses \eqref{globalhypo0} and \eqref{globalhypo1} in the described online setting in a open-end scenario we propose a monitoring scheme defined by
\begin{align}\label{eq:hatE}
\hatE_m(k) = m^{-1/2}\max_{j=0}^{k-1} (k - j) \Big\Vert \hattheta_{1}^{m+j} - \hattheta_{m+j+1}^{m+k} \Big\Vert_{\hatSigma^{-1}}~,
\end{align}
where the statistic $\hatSigma$ denotes an estimator of the long-run variance matrix $\Sigma$ (defined in Assumption~\ref{assump:approx}) and the symbol $\Vert v \Vert_A^2 = v^\top A v$ denotes a weighted norm of the vector $v$ induced by the positive definite matrix $A$.
The monitoring is then performed as follows.
With observation $X_{m+k}$ arriving, one computes $\hatE_m(k)$ and compares it to an appropriate weight function, which is sometimes also called \textit{threshold function}, say $w$.
If
\begin{align}\label{ineq:reject}
w(k/m)\hatE_m(k) > c(\alpha)
\end{align}
occurs, monitoring is stopped and the null hypothesis \eqref{globalhypo0} is rejected in favor of the alternative \eqref{globalhypo1}.
If the inequality \eqref{ineq:reject} does not hold, monitoring is continued with the next observation $X_{m+k+1}$.
We will derive the limiting distribution of $\sup_{k=1}^{\infty} \hatE_m(k)w(k/m)$ in Theorem \ref{thm:mainH0} below to determine the constant $c(\alpha)$ involved in \eqref{ineq:reject}, such that the test keeps a nominal level of $\alpha$ (asymptotically as $m \to \infty$).

\begin{remark}\label{rem:JASA-paper}
The statistic \eqref{eq:hatE} is related to a detection scheme, which was recently proposed by \cite{Dette2019} for the closed-end case, where monitoring ends with observation $mT$, for some $T \in \mathbb{N}$.
These authors considered the statistic
\begin{align}\label{eq:hatD}
\begin{split}
\hatD_m(k)
&= m^{-3/2}\max_{j=0}^{k-1} (m + j)(k - j) \Vert \hattheta_{1}^{m+j} - \hattheta_{m+j+1}^{m+k} \Vert_{\hatSigma^{-1}}~,
\end{split}
\end{align}
and showed
\begin{align}\label{h6} 
\max_{k=1}^{mT} w(k/m)\hatD_m(k) \convd \max_{t \in [0,T]}w{(t)}\max_{s \in [0,t]} |(s+1)W(t+1)- (t+1)W(s+1)|~,
\end{align}
where $W$ denotes a $p$-dimensional Brownian motion and throughout this paper 
the symbol $ \convd $ denotes weak convergence (in the space under consideration).
[To avoid confusion, note that in the reference \cite{Dette2019} the weight function was defined as $w = \dfrac{1}{\overline{w}}$ for some appropriate function $\overline{w}$].
However, this statistic cannot be considered in an open-end scenario for the typical weight functions considered in the literature satisfying  $\limsup_{t \to \infty} tw(t) < \infty$ (in this case the limit on the right-hand side
of \eqref{h6}  would be almost surely infinite for $T=\infty$).
As weight functions satisfying $\limsup_{t \to \infty} t^2w(t) < \infty$ will cause a loss in power as indicated in an unpublished simulation study, we propose to replace the factor $(m+j)$ in \eqref{eq:hatD} by the size of the initial sample $m$, which leads to the monitoring scheme defined by \eqref{eq:hatE}.
The remaining weight factor $(k-j)$ is retained as it allocates smaller weights to the case when the post-change estimator $\hattheta_{m+j+1}^{m+k}$ contains greater uncertainty as $j$ is close to $k$.
\end{remark}

\begin{remark}
An essential disadvantage of closed-end scenarios as considered in \cite{Dette2019} is the problem of choosing the end-point of monitoring before the procedure is launched.
This problem drops out when open-end scenarios are employed, where monitoring can (theoretically) proceed forever if no change has been detected.
Even if the statistical problems of closed- and open-end scenarios are naturally related, the reader should note, that the mathematical/technical access to both problem is completely different.
In the closed-end case it is usually sufficient to assume the existence of functional central limit theorems (FCLTs) as the underlying time frame is compact [see for instance \citep{Aue2012}, \cite{Wied2013}, \cite{Pape2016}, \cite{Dette2019}].
To the authors best knowledge, an FCLT is insufficient in the open-end case and one commonly assumes stronger, uniform stochastic approximations or combines an FCLT with H\'{a}y\'{e}k-R\'{e}yni type inequalities
[see also Section \ref{sec2}, \cite{Horvath2004}, \cite{Aue2009}, \cite{Aue2009b}, \cite{Fremdt2014}, \cite{Fremdt2015}, \cite{Kirch2018}].
\end{remark}

To discuss the asymptotic properties of our approach, we require the following notation.
We denote the non-negative reals by $\R_{\geq 0}$ and define $\R_{+} := \R_{\geq 0} \setminus \{0\}$.
The symbol $\convp$ denotes convergence in probability.
The process $\{W(s)\}_{s \in [0,\infty)}$ will represent a standard $p$-dimensional Brownian motion with independent components.
For a vector $v \in \R^d $, we denote by $|v| = \big ({\sum_{i=1}^d v_{i}^2} \big) ^{1/2}$ its Euclidean norm.
By $\floor{x}$ for $x \in \R$ we denote the largest integer smaller or equal to $x$.
For the sake of a clear distinction we will employ
$\sup\limits_{i=1}^n a(i)$
for discrete indexing (with integer arguments) and
$\sup\limits_{0 \leq x \leq 1} a(x)$
for continuous indexing (with arguments taken from the interval $[0,1]$ or another subset of $\R$).\\
Next, we define the influence function (assuming its existence) by
\begin{align}\label{def:inflfunc}
\IF(x,F,\theta)
= \lim_{\varepsilon \searrow 0} \dfrac{\theta((1-\varepsilon)F + \varepsilon\delta_x) - \theta(F)}{\varepsilon}~,
\end{align}
where $\delta_x(z) = I\{x \leq z\}$ is the distribution function of the Dirac measure at the point $x \in \R^d$ and the inequality in the indicator is again understood component-wise.
We will focus on functionals that allow for an asymptotic linearization in terms of the influence function, that is
\begin{align}\label{eq:meta-remainder}
\hattheta_i^j - \theta = \theta( \hat F_i^j ) -
\theta(F) = \dfrac{1}{j-i+1} \sum_{t=i}^j \IF(X_t, F, \theta) + R_{i,j}
\end{align}
with asymptotically negligible remainder terms $R_{i,j}$.
Finally, for the sake of readability we introduce the following abbreviation
\begin{align*}
\IF_t = \IF(X_t, F_t, \theta)~,
\end{align*}
where $F_t$ is again the distribution function of $X_t$.
Under the null hypothesis~\eqref{globalhypo0} we will impose the following assumptions on the underlying time series.

\begin{assump}[Approximation]\label{assump:approx}
The time series $\{X_t\}_{t \in \Z}$ is (strictly) stationary, such that $F_t = F$ for all $t \in \Z$.
Further, for each $m \in \N$ there exist two independent,\\ $p$-dimensional standard Brownian motions $W_{m,1}$ and $W_{m,2}$, such that for some positive constant $\xi < 1/2$ the following approximations hold
\begin{align}\label{eq:assumpapprox1}
\supkinf \dfrac{1}{k^\xi}
\bigg|\sum_{t=m+1}^{m+k} \IF_t -\sqrt\Sigma W_{m,1}(k)\bigg|= \Op(1)
\end{align}
and
\begin{align}
\dfrac{1}{m^\xi}
\bigg|\sum_{t=1}^{m} \IF_t -\sqrt\Sigma W_{m,2}(m)\bigg|= \Op(1)
\end{align}
as $m \to \infty$, where $\Sigma= \sum_{t \in \mathbb{Z}}
\Cov\big(\IF_0,~\IF_t\big) \in \R^{p\times p}$ denotes the long-run variance matrix of the process $\big\{\IF_t\big\}_{t \in \mathbb{Z}}$, which we assume to exist and to be non-singular.
\end{assump}

\begin{assump}[Weight function]\label{assump:weighting}
The weight function $w: \R_{\geq 0} \to \R_{\geq 0}$ is of the form
\begin{align}\label{def:cutoffsweighting}
w(t) = \tilde{w}(t)I\{ t_w \leq t \leq T_w\}
\end{align}
for $t_w \geq 0$ and $T_w \in \R_{+} \cup \{\infty\}$.
Further $\tilde{w}:\R_{\geq 0} \to \R_+$ is a positive continuous function and in case of $T_w=\infty$ it satisfies additionally
\begin{enumerate}[(1)]
\item $\limsup_{t \to \infty} t \tilde{w}(t) < \infty~,$
\item $1/\tilde{w}$ is uniformly continuous on $\R_{\geq 0}$~.
\end{enumerate}
\end{assump}

\begin{assump}[Linearization]\label{assump:remainder}
The remainder terms in the linearization \eqref{eq:meta-remainder} satisfy
\begin{align}\label{eq:remainderZero}
\maxijk \dfrac{(j-i+1)}{\sqrt{k}} | R_{i,j} | 
= o(1)
\end{align}
as $k \to \infty$ with probability one.
\end{assump}

\begin{remark}\label{rem:assump}
Let us give a brief explanation on the assumptions stated above.
\begin{enumerate}[(i)]
\item Assumption \ref{assump:approx} is a uniform invariance principle and frequently used in the (sequential) change point literature [see for example \cite{Aue2006} or \cite{Fremdt2015} among others].
Following the lines of \cite{Aue2006} Assumption \ref{assump:approx} can be verified by employing the multivariate strong approximation results derived by \cite{Eberlein1986}.
This is already spelled out for augmented GARCH-processes in Lemma A.1 of \cite{Aue2006} for the one-dimensional case.
Assumption \ref{assump:approx} is stronger than a functional central limit theorem (FCLT), which is usually sufficient to work with in a closed-end setup [see for example \cite{Wied2013}, \cite{Pape2016} or \cite{Dette2019}].
Another possible starting point to cope with open-end scenarios is an FCLT for any fixed time horizon together combined with H\'{a}y\'{e}k-R\'{e}yni-Inequalities [see for example \cite{Kirch2018} or \cite{Kirch2019}].
As this is less frequently used in the literature, we will remain with the other approach.
\item Assumption \ref{assump:weighting} gives restrictions on the feasible set of weight functions, which are required for the existence of a (weak) limit derived in Theorem \ref{thm:mainH0}.
The cutoffs defined in \eqref{def:cutoffsweighting} serve only for technical purposes.
By choosing $t_w>0$ a delay at monitoring start is introduced, which can avoid problems with false alarms due to instability [see also \cite{Kirch2018}].
Selecting $T_w<\infty$ allows to additionally cover closed-end scenarios by our theory, which we briefly discuss in Section \refb{sec:ClosedEnd} of the online appendix [see \cite{Goesmann2020appendix}].
Note that in case of $t_w=0$ and $T_w=\infty$ the cutoffs disappear, such that $w$ and $\tilde{w}$ coincide.
\item 
It is worth mentioning that it is also possible to define the functions $w,\tildew$ on the smaller domain $\R_{+}$, while
additionally demanding that $ \lim_{t \to 0} t^\gamma \tildew(t) = 0 $ for a constant $0\leq \gamma < 1/2$.
In this case, the assumption for the remainders in \eqref{eq:remainderZero} has to be replaced by
\begin{align*}
\max_{\substack{i,j=1\\ i < j}}^{k} \dfrac{(j-i+1)}{k^{1/2-\gamma}} | R_{i,j} |~
= o(1)\;\;a.s.~,
\end{align*}
which would have the upside to allow for an unbounded weighting at zero.
However, for the sake of a transparent presentation, we use Assumption \ref{assump:weighting} here, as this also simplifies the technical arguments in the proofs later on.
\item Assumption \ref{assump:remainder} is crucial for the proof of our main theorem and directly implies
\begin{align*}
\sup_{k=1}^\infty \max_{\substack{i,j=1\\ i < j}}^{m+k} \dfrac{(j-i+1)}{(m+k)^{1/2}} | R_{i,j} |
= \sup_{k=m+1}^\infty \maxijk \dfrac{(j-i+1)}{k^{1/2}} | R_{i,j} | 
= o(1)\;\;a.s.\;\text{as } m\to \infty~.
\end{align*}
Note that in the location model $\theta(F) = \E_F[X]$ we have $R_{i,j}=0$ and \eqref{eq:remainderZero} obviously holds.
In general however, Assumption \ref{assump:remainder} is highly non-trivial and crucially depends on the structure of the functional $\theta$ and the time series $\{X_t\}_{t \in \Z}$.
For a comprehensive discussion the reader is referred to \cite{Dette2019}, where the estimate \eqref{eq:remainderZero} has been verified in probability for different functionals including quantiles and variance.
\end{enumerate}
\end{remark}
\noindent The following result is the main theorem of this section.

\begin{theorem}\label{thm:mainH0}
Assume that the null hypothesis \eqref{globalhypo0} and Assumptions \ref{assump:approx} - \ref{assump:remainder} hold.
If further $\hatSigma$ is a consistent and non-singular estimator of the long-run variance matrix $\Sigma$, it holds that
\begin{align}\label{eq:ThmMainH01}
\begin{split}
\supkinf w(k/m) \hatE_m(k)
&\convd
\sup_{0 \leq t < \infty} \max_{0 \leq s \leq t} (t+1)w(t) \Big| W\Big(\dfrac{s}{s+1}\Big) - W\Big(\dfrac{t}{t+1}\Big) \Big|~,
\end{split}
\end{align}
as $m \rightarrow \infty$, where $W$ is a $p$-dimensional Brownian motion with independent components and $|\cdot|$ denotes the Euclidean norm.
\end{theorem}

For the sake of completeness, the reader should note that due to Assumption \ref{assump:weighting} the asymptotic behaviour of the weight function guarantees that the random variable on the right-hand side of \eqref{eq:ThmMainH01} is finite (with probability one).

\noindent In light of Theorem \ref{thm:mainH0} one can choose a constant $c(\alpha)$, such that
\begin{align}\label{ineq:calpha}
\Pb \bigg( \sup_{0 \leq t < \infty} \max_{0 \leq s \leq t} 
(t+1)w(t) \Big| W\Big(\dfrac{s}{s+1}\Big) - W\Big(\dfrac{t}{t+1}\Big) \Big|
> c(\alpha)\bigg) \leq \alpha~.
\end{align}

Note that for Theorem~\ref{thm:mainH0} we only require that $\hatSigma$ is a consistent estimator for the long-run variance (LRV) as $m \to \infty$. 
Under both, $H_0$ and $H_1$, such an estimator should be computed from the initial stable set, which prevents the estimate from being corrupted by possible changes/breaks [see also the discussion in Section \ref{sec4}].
In practice, the actual choice of LRV-estimator depends on the concrete application and is crucial for the performance of the procedure. 
A more extensive discussion on LRV-estimation (not only for change point problems) can be found in \cite{Andrews1991} or \cite{Shao2010}.

The following corollary then states that our approach leads to a level $\alpha$ detection scheme.
\begin{corollary}\label{cor:level}
Grant the assumptions of Theorem \ref{thm:mainH0} and further let $c(\alpha)$ satisfy inequality \eqref{ineq:calpha}, then 
\begin{align*}
\limsup_{m \to \infty}\; \Pb \bigg( \supkinf w(k/m)\hatE_m(k)> c(\alpha)\bigg) \leq \alpha~.
\end{align*}
\end{corollary}

The limit distribution obtained in Theorem \ref{thm:mainH0} strongly depends on the considered weight function. 
A special family of functions that has received considerable attention [see \cite{Horvath2004}, \cite{Fremdt2015}, \cite{Kirch2018} among many others] is given by
\begin{align}\label{eq:threshold}
w_\gamma(t) 
= (1+t)^{-1}\max\Big\{ \Big(\dfrac{t}{1+t}\Big)^\gamma,\, \varepsilon \Big\}^{-1}
\qquad\text{with}\qquad 0\leq \gamma < 1/2~,
\end{align}
where the cutoff $\varepsilon>0$ can be chosen arbitrary small in applications and only serves to reduce the assumptions and technical arguments in the proof [see also \cite{Wied2013}].
With these functions the limit distribution in~\eqref{eq:ThmMainH01} can be simplified to an expression that is more easily tractable via simulations.
Straightforward calculations yield that Assumption \ref{assump:weighting} is satisfied by the function $w_\gamma\,$ and the limit distribution in Theorem~\ref{thm:mainH0} simplifies as follows.
\begin{corollary}\label{cor:simplify}
For a $p$-dimensional Brownian motion $W$ with independent components it holds that
\begin{align*}
\sup_{0 \leq t < \infty} \max_{0 \leq s \leq t} (t+1)w_\gamma(t) \Big| W\Big(\dfrac{s}{s+1}\Big) - &W\Big(\dfrac{t}{t+1}\Big) \Big|\\
&\eqd \sup_{0 \leq t < 1} \max_{0 \leq s \leq t} \dfrac{1}{\max\{ t^\gamma, \varepsilon\}} \Big| W(t) - W(s) \Big|
:= L_{1,\gamma}~.
\end{align*}
\end{corollary}

\begin{remark}\label{rem:exactcdf1}
The cumulative distribution function of the random variable on the right-hand side in Corollary \ref{cor:simplify} is hard to derive in general.
However in the case of $\gamma=0$ and dimension $p=1$, an explicit formula can be obtained.
Therefor note that (if we ignore the cutoff constant $\varepsilon$) the following identity holds with probability one
\begin{align*}
\sup_{0 \leq t < 1} \max_{0 \leq s \leq t} \Big| W(t) - W(s) \Big|
= \max_{0 \leq t \leq 1} W(t) - \min_{0 \leq t \leq 1} W(t)~,
\end{align*}
where the distribution on the right-hand side is known as the \textit{Range of a Brownian motion} [see for instance \cite{Feller1951}].
Its distribution function can be found in \citet[p.~146]{Borodin1996} and is given by
\begin{align}\label{eq:BorodinFormula}
F_{L_1,\gamma=0}(x)
= 1 + 8 \sum_{k=1}^\infty (-1)^k \cdot k  \cdot \big(1- \Phi(kx)\big)~,
\end{align}
where $\Phi$ denotes the c.d.f. of a standard Gaussian random variable.
A corresponding result holds for the limit distribution in a closed-end scenario, see Section \refb{sec:ClosedEnd} of the online appendix [see \cite{Goesmann2020appendix}], where an additional parameter in the distribution function is associated with the monitoring length.
\end{remark}

\medskip

\noindent 
For the investigation of the consistency of the monitoring scheme \eqref{eq:hatE} we require the following assumption.
\begin{assump}\label{assump:alt}
Under the alternative $H_1$ defined in \eqref{globalhypo1} let 
\begin{align*}
\theta^{(1)}_m := \theta(F_1)=\theta(F_2)=\dots=\theta(F_{m+k^*_m-1})
\neq 
\theta^{(2)}_m := \theta(F_{m+k_m^*})=\theta(F_{m+k^*_m+1})=\cdots~,
\end{align*}
where the position of the change within the monitoring data $k_m^*\in \N$ may depend on $m$.
For the size of change suppose that
\begin{align*}
\sqrt{m} \Big| \theta^{(1)}_m - \theta^{(2)}_m \Big| \convm \infty~.
\end{align*}
Further assume that the process $\{\IF_t\}_{t \in \Z}$ and the remainders defined in Assumption \ref{assump:remainder} are of the following order before the change
\begin{align}\label{alt:beforeChange}
\dfrac{1}{\sqrt{m+k^*_m}} \bigg|\sum_{t=1}^{m+k_m^*-1} \IF_t \bigg|
= \Op(1)
\;\;\;
\text{and}
\;\;\;
\sqrt{m+k_m^*}| R_{1,m+k_m^*-1} | = \Op(1)~.
\end{align}
For the period following the change, assume that there exists a constant $c_a>0$ and distinct two cases:
\begin{enumerate}[(1)]
\item If $k^*_m/m = \mathcal{O}(1)$, suppose that
\begin{align}\label{alt:afterChange1}
\dfrac{1}{\sqrt{m}} \bigg|\sum_{t=m+k_m^*}^{m+k_m^*+\floor{c_am}} \IF_t \bigg|
= \Op(1)
\;\;\;
\text{and}
\;\;\;
\sqrt{m} \Big| R_{m+k_m^*,m+k_m^*+\floor{c_am}} \Big|
= \Op(1)
\end{align}
and for the cutoff constants in \eqref{def:cutoffsweighting}, assume that $t_w <  k_m^*/m + c_a \leq T_w$~.
\item If $k^*_m/m \to \infty$, suppose that
\begin{align}\label{alt:afterChange2}
\dfrac{1}{\sqrt{k_m^*}} \bigg|\sum_{t=m+k_m^*}^{m+k_m^*+\floor{c_ak_m^*}} \IF_t \bigg|
= \Op(1)
\;\;
\text{and}
\;\;
\sqrt{k_m^*} \Big| R_{m+k_m^*,m+k_m^*+\floor{c_ak_m^*}}\Big|
= \Op(1)~.
\end{align}
Assume additionally that the weight function satisfies $T=\infty$ and
\begin{align}\label{lim:thresholdExtra}
\liminf\limits_{t \to \infty} t\tilde{w}(t) > 0~.
\end{align}
\end{enumerate}
\end{assump}
\begin{remark}
The assumptions stated above are substantially weaker than those used to investigate the asymptotic properties of $\sup_{k=1}^{\infty} w(\tfrac{k}{m})\hatE_m(k)$ under the null hypothesis.
Basically, we only assume reasonable behavior of the time series before and after the change point and can drop the uniform approximation in Assumption \ref{assump:approx} and the uniform negligibility of the remainders in Assumption \ref{assump:remainder}.
It is easy to see, that the conditions on the sequence $\IF_t$ are already satisfied if both, its phases before and after the change fulfill a central limit theorem.
Finally, it is worth mentioning that the assumptions for the change position $k^*_m$ and size $| \theta^{(1)}_m - \theta^{(2)}_m| $ are very flexible as we allow both quantities to depend on $m$, where the latter can also tend to zero (sufficiently slow as $m \to \infty$).

For \textit{early} changes, that is $k_m^*/m = \mathcal{O}(1)$, it is obvious that the change has to occur before monitoring is stopped, where the inequality $k^*_m/m \leq T_w-c_a$ ensures that there is enough data, such that it can actually be detected.
On the other hand, the motivation for the inequality $t_w <  k_m^*/m + c_a$ is slightly more technical.
Roughly spoken, it guarantees, that the time frame $m+k^*_m,\dots,m+k_m^*+c_am$, which follows the change, is not completely covered by the weight function's cutoff at monitoring start.
For exactly this time frame we know by assertion \eqref{alt:afterChange1}, that the time series still behaves reasonable. 

For \textit{late} changes, that is $k^*_m/m \to \infty$, it is by Assumption \ref{assump:alt} not allowed to use a cutoff ($T_w<\infty$) in the weight function.
Here we rely on the extra assumption in \eqref{lim:thresholdExtra}, which defines a lower bound for the growth rate of the weight function.
Heuristically, this is necessary as it guarantees, that a sufficient amount of weight is assigned even to late time points.
The reader should note that this assumption is obviously fulfilled by the standard weighting defined in \eqref{eq:threshold}.
\end{remark}
\noindent The next theorem yields consistency under the alternative hypothesis.

\begin{theorem}\label{thm:mainH1}
Assume that the alternative hypothesis \eqref{globalhypo1} and Assumptions \ref{assump:weighting} and \ref{assump:alt} hold.
If further $\hatSigma$ is non-singular and weakly convergent to a non-singular, deterministic matrix, it holds that
\begin{align*}
\supkinf w(k/m)\hatE_m(k) \convp \infty~.
\end{align*}
Consequently,
$
\lim\limits_{m \to \infty} \Pb\Big( \sup\limits_{k=1}^\infty w(k/m)\hatE_m(k) > c \Big) = 1
$
holds for any constant $c \in \R$.
\end{theorem}

\section{Some specific change point problems} \label{sec3}
In this section we briefly illustrate how the theory developed in Section \ref{sec2} can be employed to construct monitoring schemes for a specific parameter of the distribution function.
For the sake of brevity we restrict ourselves to the mean and the parameters in a linear model.
Other examples such as the variance or quantiles can be found in \cite{Dette2019}.

\subsection{Changes in the mean}\label{sec:mean}
The sequential detection of changes in the mean
\begin{align*}
\mu(F) = \E_F[X] = \int\limits_{\R^d} x dF(x)
\end{align*}
has been extensively discussed in the literature [see \cite{Aue2004}, \cite{Fremdt2014} or \cite{Hoga2017} among many others].

\noindent Is is easy to verify (and well known), that the influence function for the mean is given by
\begin{align*}
\IF(x, F, \mu) = x - \E_F[X]~,
\end{align*}
and Assumption \ref{assump:remainder} and the corresponding parts of Assumption \ref{assump:alt} are obviously satisfied in this case since we have $R_{i,j}=0$ for all $i,j$.
For the remaining assumptions in Section 2 it now suffices that the centered time series $\big\{X_t -\E[X_t]\big\}_{t \in \Z}$ fulfills Assumption \ref{assump:approx}, which also implies the remaining part of Assumption \ref{assump:alt} [see also the discussion in Remark \ref{rem:assump}].
In this situation both, Theorem \ref{thm:mainH0} and Theorem \ref{thm:mainH1} are valid provided that the chosen weighting fulfills Assumption \ref{assump:weighting}.

\subsection{Changes in linear models}\label{sec:linearmodel}
Consider the time-dependent linear model
\begin{align}\label{eq:LM}
Y_t = P_t^\top\beta_t + \varepsilon_t~,
\end{align}
where the random variables $\{P_t\}_{t \in \N}$ are the $\R^{p}$-valued predictors, $\beta_t \in \R^p$ is a $p$-dimensional parameter and $\{\varepsilon_t\}_{t \in \N}$ is a centered random sequence independent of $\{P_t\}_{t \in \N}$.
The identification of changes in the vector of parameters in the linear model represents the prototype problem in sequential change point detection as it has been extensively studied in the literature [see \cite{Chu1996}, \cite{Horvath2004}, \cite{Aue2009}, \cite{Fremdt2015}, among many others].

This situation is covered by the general theory developed in Section \ref{sec2} and \ref{sec3}.
To be precise let
\begin{align}\label{eq:defXLM}
X_t = (P_t^\top, Y_t)^{\top} \;\in \R^d~, \;\;d=p+1\;\;\text{and}\;\; t=1,2\dots
\end{align}
be the joint vectors of predictor and response with (joint) distribution function $F_t$, such that the marginal distributions of $Y_t$ and $P_t$ are given by
\begin{align*}
F_{t,Y} = F_t(\infty,\dots,\infty,\cdot)
\;\;\;
\text{and}
\;\;\;
F_{t,P} = F_t(\cdot,\dots,\cdot,\infty)~,
\end{align*}
respectively, where we will assume that the predictor sequence is strictly stationary, that is $F_{t,P} = F_P$.
In a first step we will consider the case, where the moment matrix
\begin{align*}
M:= \E[P_1P_1^\top] = \int_{\R^d} \rho \cdot \rho^\top dF_P(\rho)
\end{align*}
is known (we will discuss later on why this assumption is non-restrictive) and non-singular.
In this setup, the parameter $\beta_t$ can be represented as a functional of the distribution function $F_t$, that is
\begin{align*}
\beta_t= \beta(F_t)
:= M^{-1}\cdot \int_{\R^d}\rho \cdot y dF_t(y,\rho)
= M^{-1}\cdot \E\big[P_tY_t\big]~,
\end{align*}
which leads to the estimators
\begin{align}\label{eq:estimatorLM}
\hatbeta_i^j 
=\beta(\hatFij)
&= \dfrac{M^{-1}}{j-i+1}\sum_{t=i}^jP_tY_t
\end{align}
from the sample $(P_i,Y_i),\dots,(P_j,Y_j)$.
To compute the influence function, let $(\rho,y) \in \R^p \times \R$, then
\begin{align*}
\begin{split}
\IF\big((\rho,y),F_t,\beta\big) 
&= \lim_{\eta \searrow 0} \dfrac{\beta\big((1-\eta)F_t + \eta \delta_{(\rho,y)}\big) - \beta(F_t)}{\eta}\\
&= \lim_{\eta \searrow 0}
M^{-1}
\Bigg[\dfrac{(1-\eta)\E[P_tY_t] + \eta \rho y}{\eta}\Bigg]
-\dfrac{\beta_t}{\eta}
= M^{-1}\big(\rho y - \E[P_tY_t] \big)~,
\end{split}
\end{align*}
which is the influence function (for $\beta$) in the linear model stated above [see for example \cite{Hampel1986} for a comprehensive discussion on influence functions].
In the following, we will use the notation $\IF_t = \IF\big(X_t,F_t,\beta\big)$ again.
Note that
\begin{align}\label{eq:infLM}
\IF_t
= M^{-1}\big(P_tY_t - \E[P_tY_t] \big) = M^{-1}P_tY_t - \beta_t ~,
\end{align}
which directly gives $\E[\IF_t]=0$.
Assuming additionally stationarity of $\{\varepsilon_t\}_{t \in N}$, it follows that the random sequence $\{X_t\}_{t\in \N}$ is stationary under the null hypothesis.
In this case, the linearization defined in \eqref{eq:meta-remainder} simplifies to
\begin{align}\label{eq:linearizationLM}
\begin{split}
\hatbeta_i^j - \beta_1
=\beta(\hatFij) - \beta_1
&= \dfrac{M^{-1}}{j-i+1}\sum_{t=i}^jP_tY_t - \beta_1 
= \dfrac{1}{j-i+1}\sum_{t=i}^j\big(M^{-1}P_tY_t -\beta_t\big)\\
&= \dfrac{1}{j-i+1}\sum_{t=i}^j \IF_t~.
\end{split}
\end{align}
Consequently,  the remainders in \eqref{eq:meta-remainder}  vanish and Assumption \ref{assump:remainder} is obviously satisfied.
Next, note that the long-run variance matrix is given by
\begin{align}\label{eq:lrvLM}
\Sigma
= \sum_{t \in \mathbb{Z}}
\Cov\big(\IF_0,~\IF_t\big)
= M^{-1}\Gamma M^{-1}
\end{align}
with $\Gamma = \sum_{t \in \mathbb{Z}}\Cov\big(Y_0P_0,~Y_tP_t\big)$, which can be estimated by $\hatSigma= M^{-1}\hatGamma M^{-1}$ where $\hatGamma$ is an estimator for $\Gamma$.
Observing \eqref{eq:linearizationLM} it is now easy to see that 
in the resulting  statistic $\hatE_{m}$ the matrix $M$ cancels out, that is 
\begin{align}\label{eq:rewritehatE}
\begin{split}
\hatE_m(k)
&= m^{-1/2}\max_{j=0}^{k-1} (k - j) \Big\Vert \hatbeta_{1}^{m+j} - \hatbeta_{m+j+1}^{m+k} \Big\Vert_{\hatSigma^{-1}}\\
&= m^{-1/2}\max_{j=0}^{k-1} (k - j) \Big\Vert \dfrac{1}{m+j}\sum_{t=1}^{m+j} Y_tP_t - \dfrac{1}{k-j}\sum_{t=m+j+1}^{m+k} Y_tP_t \Big\Vert_{\hatGamma^{-1}}
\end{split}
\end{align}
and for this reason it does not depend on the matrix $M$.
We therefore obtain the following result, which describes the asymptotic properties of the monitoring scheme based on the statistic $\hat E_{m}$ for a change in the parameter in the linear regression model
\eqref{eq:LM}. 
The proof is a direct consequence of Theorems \ref{thm:mainH0} and \ref{thm:mainH1}.

\begin{corollary}\label{cor:LM1}
Assume that the predictor sequence $\{P_t\}_{t \in \N}$ and the centered sequence $\{\varepsilon_t\}_{t \in \N} $ are strictly stationary and the second moment matrix $M=\E[P_1P_1^\top]$ is non-singular.
Further suppose that the sequences $\{P_t\}_{t \in \N}$ and $\{\varepsilon_t\}_{t \in \N}$ are independent and let the weight function under consideration fulfill Assumption \ref{assump:weighting}.
\begin{enumerate}[(i)]
\item Under the null hypothesis $H_0$ of no change, it follows that the sequence $\{\IF_t\}_{t \in \N}$ defined in \eqref{eq:infLM} is strictly stationary.
Assume further that this sequence admits the approximation in Assumption \ref{assump:approx} and that $\hatGamma_m$ is a non-singular, consistent estimator of the non-singular long-run variance matrix $\Gamma$ defined in \eqref{eq:lrvLM}.
Then monitoring based on the statistic $\hatE$ in \eqref{eq:rewritehatE} is an asymptotic level $\alpha$ procedure.
\item Under the alternative hypothesis $H_1$ suppose that Assumption~\ref{assump:alt} is fulfilled.
If further $\hatGamma$ is non-singular and weakly convergent to a non-singular, deterministic matrix, the monitoring based on the statistic $\hatE$ in \eqref{eq:rewritehatE} is consistent.
\end{enumerate}
\end{corollary}

\begin{remark}\label{rem:WeightedLM}
If one replaces (the unknown) moment matrix $M$ on the right-hand side of \eqref{eq:estimatorLM} by an appropriate estimate, that is
\begin{align*}
\hat{\hat{\beta}}_i^j = \dfrac{\big(\hat{M}_i^j\big)^{-1}}{j-i+1}\sum_{t=i}^jP_t Y_t
\;\;\;\text{with}\;\;\;
\hat{M}_i^j =  \dfrac{1}{j-i+1} \sum_{t=i}^{j} P_tP_t^\top~,
\end{align*}
one obtains a modified statistic given by
\begin{align}\label{eq:modEhatLM}
\begin{split}
\hat{\hat{E}}_m(k) 
&= m^{-1/2}\max_{j=0}^{k-1} (k - j) \Big\Vert \hat{\hat{\beta}}_{m+j+1}^{m+k} - \hat{\hat{\beta}}_{1}^{m+j}\Big\Vert_{\hat{\hat{\Sigma}}_m^{-1}}\\
&= m^{-1/2}\max_{j=0}^{k-1} \Big\Vert \big(\hat{M}_{m+j+1}^{m+k}\big)^{-1}\sum_{t=m+j+1}^{m+k} P_t\big(Y_t -P_t^\top \hat{\hat{\beta}}_{1}^{m+j} \big)\Big\Vert_{\hat{\hat{\Sigma}}_m^{-1}}~,
\end{split}
\end{align}
where $\hat{\hat{\Sigma}}_m$ denotes an appropriate long-run variance estimator.
Note that in this situation the dependence of the unknown moment matrix $M$ (or its estimators) cannot cancel out as observed in \eqref{eq:rewritehatE}.
The modified statistic can be reasonable to employ if - for example - possible changes in the distribution of $P_t$ have to be taken into account.
However as equation \eqref{eq:modEhatLM} illustrates the modified statistic $\hat{\hat{E}}$ can equivalently be written as weighted residual-based approach.
This kind of phenomena is already known in the literature, as \cite{Huskova2005} describe this for a similar statistic in linear models.
\end{remark}

\section{Finite sample properties}\label{sec4}
In this section we investigate the finite sample properties of our monitoring procedure and demonstrate its superiority with respect to the available methodology.
We choose the following two statistics as our benchmark
\begin{align}\label{eq:otherStatistics}
\begin{split}
\hat{Q}_m(k) &:= \dfrac{k}{m^{1/2}} \Big\Vert \hattheta_{1}^{m} - \hattheta_{m}^{m+k} \Big\Vert_{\hatSigma^{-1}}~,\\
\hat{P}_m(k) &:= \max_{j=0}^{k-1} \dfrac{k - j}{m^{1/2}} \Big\Vert \hattheta_{1}^{m} - \hattheta_{m+j+1}^{m+k} \Big\Vert_{\hat\Sigma^{-1}}~.
\end{split}
\end{align}
The procedure based on $\hat{Q}$ was originally proposed by \cite{Horvath2004} for detecting changes in the parameters of linear models and since then reconsidered for example by \cite{Aue2012}, \cite{Wied2013} and \cite{Pape2016} for the detection of changes in the CAPM-model, correlation and variances, respectively.
A statistic of the type $\hat{P}$ was recently proposed by \cite{Fremdt2015} and has been already reconsidered by \cite{Kirch2018}.
In the simulation study we will restrict ourselves to the commonly used class of weight functions $w_\gamma$ defined in \eqref{eq:threshold}, where we set the involved, technical constant $\varepsilon=10^{-10}$ when computing the statistics.
Under the assumptions made in Section \ref{sec2}, it can be shown by similar arguments as given in Section \refb{app:technical} of the online appendix [see \cite{Goesmann2020appendix}] that
\begin{align}\label{conv:limitQ}
\supkinf w_\gamma(k/m)\hat{Q}_m(k) \convd
\sup_{0 \leq t < 1} \dfrac{| W(t) |}{\max\{ t^\gamma, \varepsilon\}} =: L_{2,\gamma}
\end{align}
and
\begin{align}\label{conv:limitP}
\supkinf w_\gamma(k/m)\hat{P}_m(k) \convd 
\sup_{0 \leq t < 1} \max_{0 \leq s \leq t} \dfrac{1}{\max\{ t^\gamma, \varepsilon\}} \Big| W(t) - \dfrac{1-t}{1-s}W(s) \Big| =: L_{3,\gamma} ~,
\end{align}
where $W$ denotes a $p$-dimensional Brownian motion.
For detailed proofs (under slightly different assumptions) of \eqref{conv:limitQ} and \eqref{conv:limitP}, the reader is relegated to \cite{Horvath2004} and \cite{Fremdt2015}, where procedures of these types are considered in the special case of a linear model.

Recall the notation of $L_{1,\gamma}$ introduced in Corollary \ref{cor:simplify}.
By \eqref{conv:limitQ}, \eqref{conv:limitP} and Corollary \ref{cor:level} the necessary critical values for the procedures $\hat{E}$, $\hat{Q}$ and $\hat{P}$ combined with weighting $w_\gamma$ are given as the $(1-\alpha)$-quantiles of the distributions $L_{1,\gamma}$, $L_{2,\gamma}$ and $L_{3,\gamma}$, respectively and can easily be obtained by Monte Carlo simulations.
The quantiles are listed in Table \ref{table:criticalp1} for dimensions $p=1$ and $p=2$ and have been calculated by $10000$ runs simulating the corresponding distributions where the underlying Brownian motions have been approximated on a grid of $5000$ points.
In Sections \ref{sec:simMean} and \ref{sec:simLM} below, we will examine the finite sample properties of the three statistics for the detection of changes in the mean and in the regression coefficients of a linear model, respectively.
All subsequent results presented in these sections are based on 1000 independent simulation runs and a fixed test level of $\alpha = 0.05$. 
\begin{table}[h]
\centering
\begin{tabular}{c|c|ccc|ccc|ccc}
&& \multicolumn{3}{c|}{$L_{1,\gamma}$} & \multicolumn{3}{c|}{$L_{2,\gamma}$} & \multicolumn{3}{c}{$L_{3,\gamma}$} \\
\hline
p &$\gamma$ \textbackslash $\alpha$ & 0.01 & 0.05 & 0.1 & 0.01 & 0.05 & 0.1 & 0.01 & 0.05 & 0.1 \\
\hline
\hline
\multirow{3}{*}{1} &0  & 3.0233 & 2.4977 & 2.2412  & 2.7912 & 2.2365 & 1.9497 & 2.8262 & 2.2599 & 1.9914\\
\cline{2-11}
&0.25& 3.1050 & 2.5975 & 2.3542 & 2.9445 & 2.3860 & 2.1060 & 2.9638 & 2.4296 & 2.1758\\
\cline{2-11}
&0.45& 3.4269 & 2.9701 & 2.7398 & 3.3015 & 2.7992 & 2.5437 & 3.3817 & 2.9241 & 2.7002\\
\hline
\hline
\multirow{3}{*}{2} &0   & 3.4022 & 2.8943 & 2.6562 & 3.2272 & 2.6794 & 2.4008 & 3.2461 & 2.6957 & 2.4266 \\
\cline{2-11}
&0.25& 3.5279 & 3.0948 & 2.7781 & 3.3322 & 2.7981 & 2.5481 & 3.3630 & 2.8433 & 2.5911 \\
\cline{2-11}
&0.45& 3.8502 & 3.3912 & 3.1509 & 3.7010 & 3.2046 & 2.9543 & 3.7467 & 3.2966 & 3.0620 \\
\end{tabular}
\caption{\it (1-$\alpha$)-quantiles of the distributions $L_{1,\gamma}$, $L_{2,\gamma}$ and $L_{3,\gamma}$ for different choices of $\gamma$ and different dimensions of the parameter.
The cutoff constant was set to $\varepsilon=0$.
The results for $L_{2,\gamma}$ and $L_{3,\gamma}$ for $p=1$ are taken from  \cite{Horvath2004} and \cite{Fremdt2015}, respectively.
The quantiles for $L_{1,0}$ for $p=1$ were computed with respect to formula \eqref{eq:BorodinFormula}.
\label{table:criticalp1}}
\end{table}

\subsection{Changes in the mean}\label{sec:simMean}
In this section we will compare the finite sample properties of the procedures based on the statistics $\hatE, \hatP$ and $\hatQ$ for changes in the mean as outlined in Section \ref{sec:mean}.
Here we test the null hypothesis of no change which is given by
\begin{align}\label{LMhypo0sim}
H_0&:\; \mu_1 = \dots = \mu_m = \mu_{m+1} =\mu_{m+2} = \ldots ~,
\end{align}
while the alternative, that the parameter $\mu_t$ changes beyond the initial data set, is defined as
\begin{align}\label{LMhypo1sim}
 H_1&:\; \exists k^{\star} \in \N:\;\;
\mu_1 = \dots =\mu_{m+k^{\star}-1} \neq \mu_{m+k^{\star}} = \mu_{m+k^{\star}+1} = \ldots \;.
\end{align}
We will consider four different data generating models, one white noise process and
three autoregressive processes with different levels of temporal dependence controlled by the AR-parameter.
To be precise we consider the models
\begin{enumerate}[(M1)]
\item $X_{t} = \varepsilon_t$~,
\item $X_{t} = 0.1X_{t-1} + \varepsilon_t$~,
\item $X_{t} = 0.5X_{t-1} + \varepsilon_t$~,
\item $X_{t} = 0.7X_{t-1} + \varepsilon_t$~,
\end{enumerate}
where $\{\varepsilon_t\}$ is an i.i.d. sequence of standard Gaussian random variables.
For the AR(1)-processes defined in models (M2)-(M4), we create a burn-in sample of 100 observations in the first place.
To simulate the alternative hypotheses, changes in the mean are added to the data, that is 
\begin{align*}
X^{(\delta)}_t = 
\begin{cases}
X_t\;\;&\text{if}\;\; t < m+k^*~,\\
X_t + \delta \;\;&\text{if}\;\; t \geq m+k^*~,\\
\end{cases}
\end{align*}
where $\delta = \E[X_{m+k^*}] - \E[X_{m+k^*-1}]$ denotes the desired change amount.
For the necessary long run variance estimation we employ the well known quadratic spectral estimator [see \citep{Andrews1991}] with its implementation in the R-package `sandwich' [see \cite{Zeileis2004}].
To take into account the possible appearance of changes, only the initial stable segment $X_1,\dots, X_m$ is used for this estimate.
This restriction is standard in the literature [see for example \cite{Horvath2004}, \cite{Wied2013}, or \cite{Dette2019} among many others], and we will briefly discuss ideas to improve this in our outlook in Section \ref{sec:ConAndOut}.
The bandwidth involved in the estimator is chosen as $\log_{10}(m)$ for models (M1) and (M2).
To take into account the stronger temporal dependence we take a bandwidth of $\log_{10}(m^4)$ for the models (M3) and (M4).

In Table \ref{tab:meanOE1} and \ref{tab:meanOE2} we display the type I error for the four time series models (M1)-(M4) and different choices of $\gamma$ in the weight functions.
The principal observation for Table \ref{tab:meanOE1} is, that all three statistical procedures offer a reasonable approximation of the desired nominal level of $\alpha=0.05$ for the models (M1) and (M2).
The results for the weak dependent model (M2) are slightly worse than those for the white noise model (M1).

In Table \ref{tab:meanOE2} it can be seen that the nominal approximation is quite imprecise for the stronger dependent models (M3) and (M4) especially for an initial sample size of $m=100$.
This effect seems to be primary caused by a less precise estimation of the long-run variance and the approximation improves with larger initial sample size $m$, such that the type I error is considerably closer to $5\%$ for $m=400$.
To support this conjecture, we also report the type I error for simulations, in which we used the true long-run variance instead of an estimate, in Table \ref{tab:meanOE2}.
This demonstrates a much more sound approximation of the desired test level of $5\%$ for $m=100$ and $m=200$.

To discuss the performance under the alternative we illustrate the power of the procedures for increasing values of the change and different change positions for models (M1) and (M2) with $m=100$ in Figure \ref{fig:OE1} and for models (M3) and (M4) with $m=200$ in Figure \ref{fig:OE2}.
As the results are very similar we only report the choice $\gamma=0$ here and provide results for $\gamma=0.45$ in the online appendix.
The basic tendency observable in Figures \ref{fig:OE1} and \ref{fig:OE2} is concordant: While the procedures behave similar for a change close to the initial data set (first row), the method based on $\hatE$ is clearly superior to the others the more the distance to the initial set grows.
\revTwo{
The advantage in power is not visible for changes occuring close to the intial training set, where the other procedures perform slightly better.%
}

To give an example, consider the right plot of the first row in Figure \ref{fig:OE1}.
Here the test based on the statistic $\hatE$ already has a power of 62.8\% for a change of $\delta=0.3$, whereas the tests based on the statistics $\hatP$ and $\hatQ$ have power of 43.7\% and 42.4\%, respectively.
The superior performance of $\hatE$ can most likely be explained by the more accurate estimate of the pre-change parameter by $\hattheta_{1}^{m+j}$, while the other statistics only involve the estimator $\hattheta_{1}^{m}$ [see formulas \eqref{eq:hatE} and \eqref{eq:otherStatistics}].

For the sake of an appropriate understanding of our findings, the reader should be aware of the fact, that - although we consider open-end procedures here - simulations have to be stopped eventually.
Here we chose this stopping point as $1000$ ($m=50)$, $3000$ ($m=100$, $m=200$) or $4000$ ($m=400$) observations and it is expectable that the testing power of all procedures increases with a later stopping point.
Therefore the observed superiority of $\hatE$ refers to \textit{the type II error until the specified stopping point}.
\revTwo{
The theory developed in Section~\ref{sec2} also covers the case with a preselected end of the monitoring period.
While the statistic for monitoring is the same, the quantile is chosen differently leading to a detector that has higher power if the change is included in the monitoring window and no power if the true change occurs after monitoring ends.
We discuss this in the online appendix.
}

\begin{table}[h]
\centering
\begin{tabular}{c|c|ccc|ccc}
& & \multicolumn{3}{c|}{(M1)} & \multicolumn{3}{c}{(M2)}\\
\hline 
\hline 
&&&&&&&\\[-1em]
$m$ & $\gamma$ & $\hatE$ & $\hatQ$ & $\hatP$ &$\hatE$ & $\hatQ$ & $\hatP$\\ 
\hline
\hline 
\multirow{3}{*}{50} 
& $0$    &  4.8\% &  5.3\% &  5.3\% &  8.4\% &  8.8\% &  9.0\% \\
\cline{2-8}
& $0.25$ &  5.0\% &  5.0\% &  5.3\% &  8.9\% &  8.4\% &  8.3\% \\
\cline{2-8}
& $0.45$ &  4.5\% &  4.4\% &  3.9\% &  7.5\% &  7.4\% &  6.4\% \\
\hline
\hline
\multirow{3}{*}{100}
& $0$    &  4.1\% &  4.4\% &  4.6\% &  6.8\% &  6.3\% &  6.6\% \\
\cline{2-8}
& $0.25$ &  5.0\% &  5.4\% &  5.6\% &  7.3\% &  6.7\% &  6.9\% \\
\cline{2-8}
& $0.45$ &  6.0\% &  6.2\% &  5.2\% &  7.0\% &  6.4\% &  6.0\% \\
\hline
\end{tabular}
\caption{\it Type I error for the open-end procedures based on $\hatE$, $\hatQ$ and $\hatP$ at 5\% nominal size.
The size of the known stable data was set to $m=50$ (upper part), $m=100$ (lower part).}
\label{tab:meanOE1}
\end{table}

\begin{table}[h]
\small
\centering
\begin{tabular}{c|c|ccc|ccc}
& & \multicolumn{3}{c|}{(M3)} & \multicolumn{3}{c}{(M4)}\\
\hline 
\hline 
&&&&&&&\\[-1em]
$m$ & $\gamma$ & $\hatE$ & $\hatQ$ & $\hatP$ &$\hatE$ & $\hatQ$ & $\hatP$\\
\hline
\hline
\multirow{3}{*}{100} 
& $0$    & 12.1\% (3.2) & 11.4\% (4.5) & 11.8\% (4.4) & 16.3\% (2.9) & 15.0\% (4.1) & 15.9\% (4.0) \\
\cline{2-8}
& $0.25$ & 13.7\% (3.2) & 11.8\% (3.8) & 12.5\% (3.8) & 18.1\% (3.1) & 16.3\% (3.5) & 17.3\% (3.5) \\
\cline{2-8}
& $0.45$ & 13.2\% (2.6) & 12.2\% (2.8) & 11.6\% (2.3) & 16.6\% (2.1) & 14.3\% (2.3) & 13.9\% (1.8) \\
\hline
\hline
\multirow{3}{*}{200}
& $0$    &  7.6\% (3.0) &  8.1\% (4.0) &  8.4\% (4.1) &  9.4\% (2.6) & 10.5\% (3.5) & 10.6\% (3.5) \\
\cline{2-8}
& $0.25$ &  8.4\% (3.5) &  7.6\% (4.0) &  8.3\% (4.1) & 11.4\% (3.0) & 11.4\% (3.8) & 11.6\% (3.8) \\
\cline{2-8}
& $0.45$ &  8.7\% (3.2) &  8.1\% (3.2) &  7.4\% (2.8) & 11.3\% (2.6) &  10.6\% (2.7) & 10.2\% (2.4) \\
\hline
\hline
\multirow{3}{*}{400}
& $0$    &  5.0\% (2.8) &  6.0\% (3.5) &  6.2\% (3.4) &  7.3\% (2.6) &  7.8\% (3.6) &  8.2\% (3.4)\\
\cline{2-8}
& $0.25$ &  6.2\% (3.5) &  6.2\% (4.0) &  6.3\% (4.1) &  8.2\% (3.1) &  8.1\% (3.8) &  8.7\% (3.6)\\
\cline{2-8}
& $0.45$ &  6.9\% (3.1) &  6.2\% (3.3) &  5.7\% (2.9) &  7.8\% (2.9) &  7.9\% (2.9) &  7.0\% (2.6)\\
\hline
\end{tabular}
\caption{\it Type I error for the open-end procedures based on $\hatE$, $\hatQ$ and $\hatP$ at 5\% nominal size.
The size of the known stable data was set to $m=100$ (upper part), $m=200$ (middle part) and $m=400$ (lower part).
In brackets we report the result of simulations, in which the long-run variance estimator has been replaced by the true long-run variance.
}
\label{tab:meanOE2}
\end{table}

\begin{figure}[H]
\includegraphics[width=15cm,height=8.85cm]{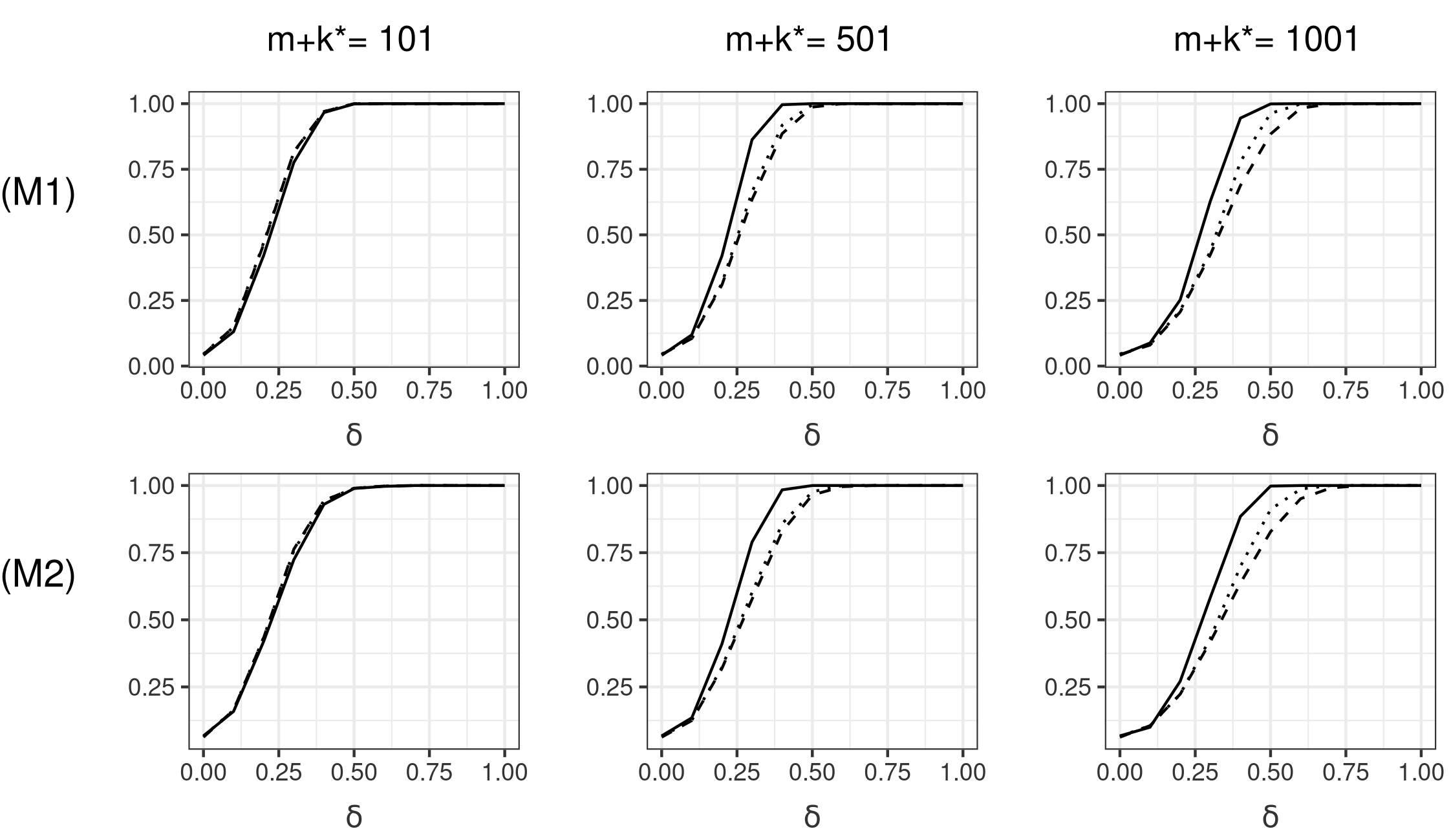} 
\vspace{-0.4cm}
\caption{\it Power of the monitoring procedures for a change in the mean based on the statistics $\hatE$ (solid line), $\hatQ$ (dashed line) and $\hatP$ (dotted line) with $\gamma=0$ and $m=100$ at 5\% nominal size.
\label{fig:OE1}}
\end{figure}

\begin{figure}[H]
\includegraphics[width=15cm,height=8.85cm]{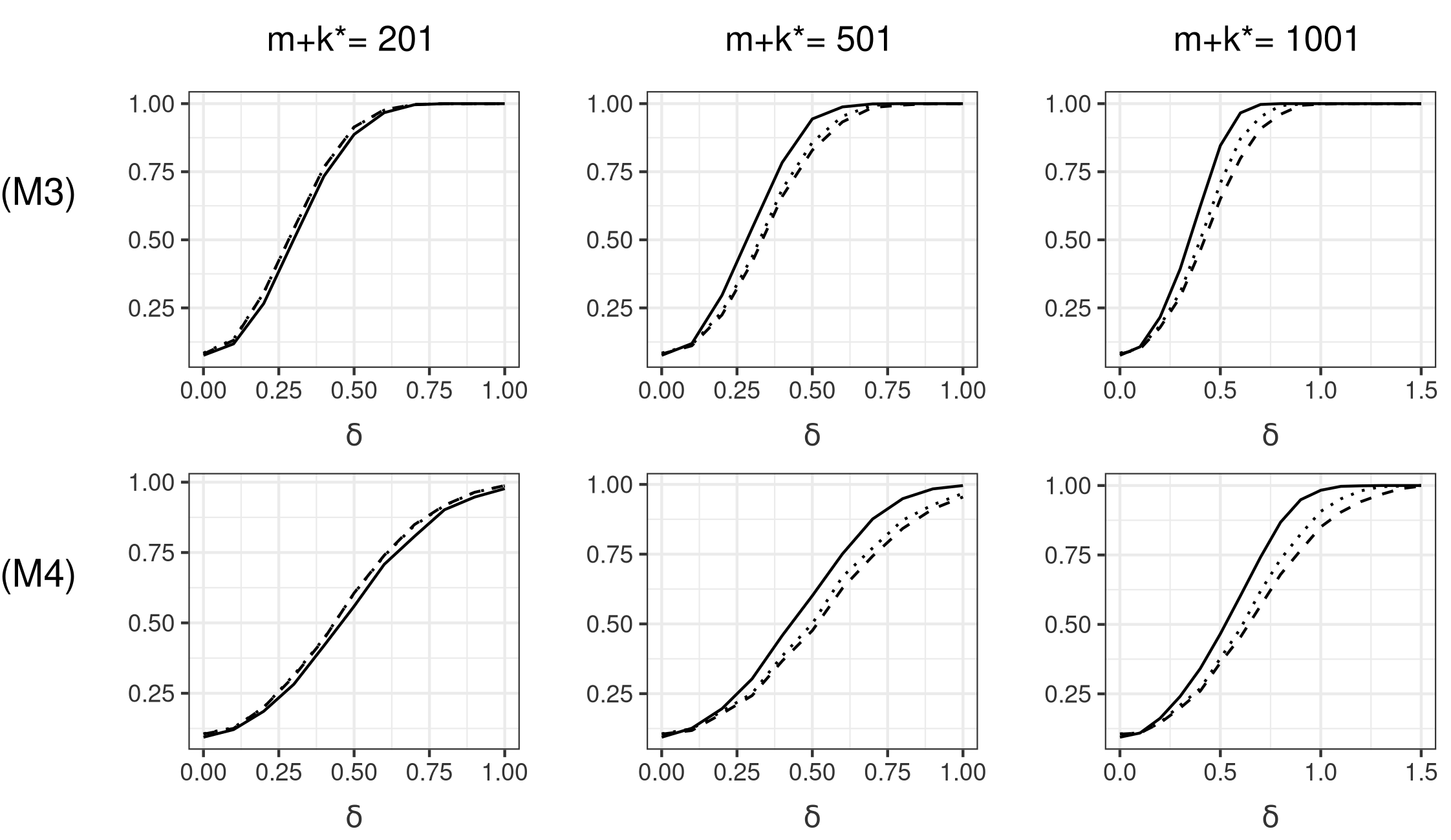} 
\caption{\it Power of the monitoring procedures for a change in the mean based on 
the statistics $\hatE$ (solid line), $\hatQ$ (dashed line) and $\hatP$ (dotted line) with $\gamma=0$ and $m=200$ at 5\% nominal size.
\label{fig:OE2}}
\end{figure}

\subsection{Changes in linear models}\label{sec:simLM}
In this section we present some simulation results for the detection of changes in the linear model \eqref{eq:LM}.
We aim to detect changes in the unknown parameter vector $\beta_t \in \R^p$ by testing the null hypothesis
\begin{align}\label{LMhypo0}
H_0&:\; \beta_1 = \dots = \beta_m = \beta_{m+1} =\beta_{m+2} = \ldots ~,
\end{align}
against the alternative that the parameter $\beta_t$ changes beyond the initial data set, that is
\begin{align}\label{LMhypo1}
 H_1&:\; \exists k^{\star} \in \N:\;\;
\beta_1 = \dots =\beta_{m+k^{\star}-1} \neq \beta_{m+k^{\star}} = \beta_{m+k^{\star}+1} = \ldots \;.
\end{align}
To be precise, we consider the model \eqref{eq:LM} with $p=2$ and the following choice of predictors
\begin{enumerate}[(LM1)]
\item $P_{t} = (1, \sqrt{0.5}Z_t)^\top$~,
\item $P_{t} = (1, 1 + G_t)^\top$ with $G_t = \bar{\sigma}_tZ_t$ and 
$\bar{\sigma}^2_t=0.5+0.2Z_{t-1}+0.3\bar{\sigma}^2_{t-1}$~,
\end{enumerate}
where $Z_t$ denotes an i.i.d. sequence of $\mathcal{N}(0,1)$ random variables in both models.
The parameter vector is fixed at $\beta_t=(1,1)$ under the null hypothesis and to examine the alternative hypothesis, changes are added to its second component, that is 
\begin{align*}
\beta^\delta_t = 
\begin{cases}
(1,1)^\top\;\;&\text{if}\;\; t < m+k^*~,\\
(1,1+\delta)^\top \;\;&\text{if}\;\; t \geq m+k^*~.\\
\end{cases}
\end{align*}
For both scenarios we simulated the residuals $\varepsilon_t$ in model \eqref{eq:LM} as i.i.d. $\mathcal{N}(0,0.5)$ sequences.
Note that the GARCH(1,1) model (LM2) has been already considered by \cite{Fremdt2015}.
As pointed out in Section \ref{sec:linearmodel} the asymptotic variance that needs to be estimated within our procedures is given by
\begin{align}\label{eq:lrvLMsim}
\Gamma
= \sum_{t \in \mathbb{Z}}\Cov\big(P_0Y_0,~P_tY_t\big)~.
\end{align}
We estimate this quantity based on the stable segment of observations $(Y_1,P_1),\dots,(Y_m,P_m)$ using the well known quadratic spectral estimator [see \cite{Andrews1991}] with its implementation in the R-package `sandwich' [see \cite{Zeileis2004}].\\

The problem of detecting changes in the parameter of the linear model has also been addressed using partial sums of the residuals
$
\hat{\varepsilon}_t = Y_t - P^\top_t\hatbeta_I~
$
in statistics similar to \eqref{eq:otherStatistics}, where $\hatbeta_I$ is an initial estimate of $\beta$ computed from the initial stable segment.
We refer for instance to \cite{Chu1996}, \cite{Horvath2004}, who - among many others - use statistics similar to $\hatQ$, or \cite{Fremdt2015}, who uses a statistic similar to $\hatP$.
Our approach directly compares estimators for the vector $\beta_t$, which are derived using the general methodology introduced in Sections \ref{sec2} and \ref{sec3}.
The resulting statistics are  obtained replacing $\hat{\theta}$ by $\hat{\beta}$ in equation \eqref{eq:otherStatistics}.
As pointed out in Remark \ref{rem:WeightedLM}, there is a strong connection between methods comparing direct estimates and methods based on weighted residuals, which was already described by \cite{Huskova2005}.
These authors, in particular, demonstrate that these approaches exhibit power against alternatives, that the \textit{plain} residual-based statistics fail to distinguish from the null hypothesis.
We also refer to \cite{Leisch2000}, \cite{Huskova2005} and \cite{Huskova2007} for a comparison of (plain) residual-based methods with methods using the estimators directly.

In Table \ref{tab:meanLM1} we display the approximation of the nominal level for the three statistics with different values of the parameter $\gamma$ in the weight function, where monitoring was stopped after $1500$ observations.
We observe an acceptable approximation of the nominal level 5\% in the case $\gamma=0$, while the rejection probabilities for $\gamma=0.25$ or $\gamma=0.45$ slightly exceed the desired level of 5\%.
The fact that larger values of $\gamma \in [0,1/2)$ can lead to a worse approximation of the desired type I error has also  been observed by other authors [see, for example,  \cite{Wied2013}] and can be explained by a more sensitive weight function at the monitoring start if $\gamma$ is chosen close to $1/2$.
Overall, the approximation is slightly better for the independent case in model (LM1).

In Figure \ref{fig:LM1} we compare the power with respect to the change amount for different change positions, where we restrict ourselves to the case $\gamma=0$ for the sake of brevity.
The results are very similar to those provided for the mean functional in Section \ref{sec:simMean}.
Again the monitoring scheme based on $\hatE$ outperforms the procedures based on $\hatQ$ and $\hatP$, and the superiority is larger for a later change.
We omit a detailed discussion and summarize that the empirical findings have indicated superiority (w.r.t. testing power) of the monitoring scheme based on the statistic $\hatE$.

\begin{table}[h]
\centering
\begin{tabular}{ccccccc}
 & \multicolumn{3}{c}{(LM1)} & \multicolumn{3}{c}{(LM2)} \\
\hline 
\hline
&&&&&&\\[-1em]
$\gamma$ & $\hatE$ & $\hatQ$ & $\hatP$ &$\hatE$ & $\hatQ$ & $\hatP$ \\ 
\hline
\hline 
$0$    &  6.4\% &  6.5\% &  6.7\% &  7.2\% &  6.7\% &  7.2\% \\
\hline
$0.25$ &  7.6\% &  8.8\% &  9.1\% &  8.5\% &  9.6\% &  9.5\% \\
\hline
$0.45$ & 12.0\% & 12.2\% & 12.1\% & 12.6\% & 12.2\% & 12.6\% \\
\hline
\end{tabular} 
\caption{\it  Type I error for the open-end procedures based on $\hatE$, $\hatQ$ and $\hatP$ at 5\% nominal size.
The size of the known stable data was set to $m=100$.}
\label{tab:meanLM1}
\end{table}

\begin{figure}[h]
\includegraphics[width=15cm,height=8.85cm]{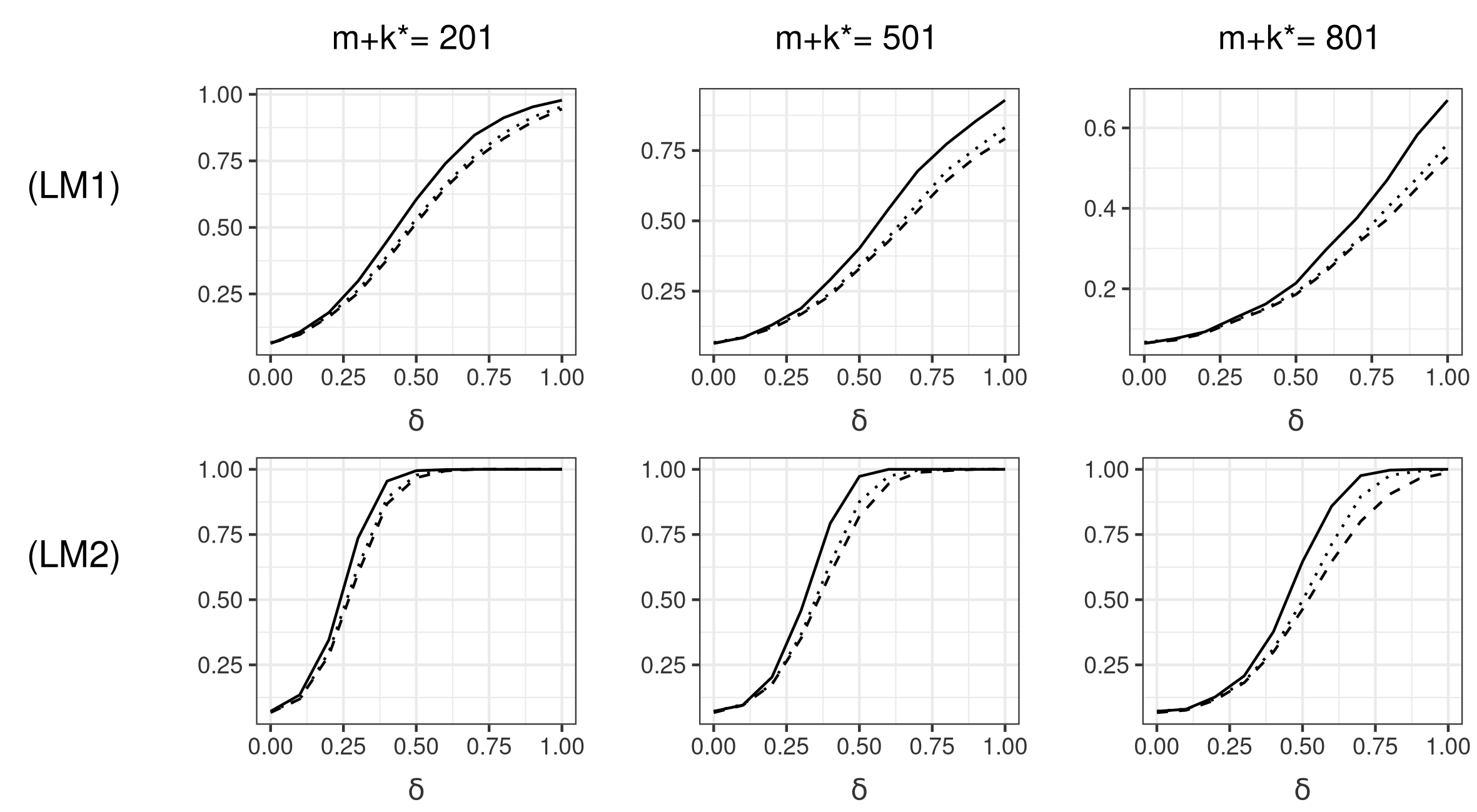} 
\caption{\it Power of the monitoring procedures for a change in the regression parameters for the open-end procedures based on the statistics $\hatE$ (solid line), $\hatQ$ (dashed line) and $\hatP$ (dotted line) with $\gamma=0$ and $m=100$ at 5\% nominal size.
\label{fig:LM1}}
\end{figure}

\section{Two applications}\label{sec:RealData}
In this section, we apply our methodology along side two competitors to monitor for changes in linear models.
We discuss two examples, related to the United Kingdom European Union membership referendum 2016.
We consider the linear model
\begin{equation}\label{eqn:LMempEx}
Y_t = \beta_1 P_{1,t} + \beta_2 P_{2,t} + \varepsilon_t~,
\end{equation}
where $Y_t$ is a real-valued response and $(P_{1,t}, P_{2,t})$ is a two-dimensional predictor, which is a special case of the linear model considered in Section~\ref{sec:linearmodel}.

Recall that our approach requires a stable segment of $m$ observations in which no changes have yet happened.
We choose a stable segment of size $m=20$ for our analysis of the data and monitor with the three detectors $\hat{E}$, $\hat{P}$ and $\hat{Q}$ defined in~\eqref{eq:hatE} and~\eqref{eq:otherStatistics}, respectively.
More precisely, the detectors are updated for every incoming observation, namely $(Y_t, P_{1,t}, P_{2,t})$, and a decision is made, by comparing the detectors with the corresponding thresholds, whether to reject the null hypothesis and stop the procedure or to continue monitoring with the subsequent observation.
Monitoring then continues until a change has been detected by each of the three approaches.

For the next monitoring phase another $m$ observations from the time where the last of the three detectors has rejected are used as the next stable segment.
Monitoring ends once the end of the available data is reached.

In the remaining part of this section, we present the outcomes of the previously described statistical analysis for two data sets related to the United Kingdom (UK) European Union (EU) membership referendum, which took place on 23 June 2016.
For our analysis we chose the significance levels to be $\alpha = 0.05$ and the weight function $w_0$, as defined in~\eqref{eq:threshold}.
All data used was obtained from \url{https://www.ariva.de} on 26 March 2020.

As our first example, we consider the relation of the UK's currency, Pound Sterling (GBP), to the Eurozone's currency, the Euro (EUR), and Switzerland's currency, the Swiss franc (CHF).
More precisely, we consider daily log returns of the exchange rate of GBP to the United States dollar (USD) as a response $Y_t$ of a linear model as described in~\eqref{eqn:LMempEx}.
As predictors we now consider the log returns of EUR to USD ($P_{1,t}$) and CHF to USD ($P_{2,t}$).
A graphical representation of the exchange rates and associated log returns for the period from Januar 2016 to December 2019 can be seen in Figure~\ref{fig:EE2}.
The outcomes of the previously described analysis are presented visually in the graphs.
The first 20 observation (4 Jan 2016 to 29 Jan 2016, note that we only considered trading days FXCM) were used as the stable segment for monitoring.
The monitoring starts on 1 Feb 2016 and went on with all three detectors until 17 Mar 2016 when $\hat{P}$ and $\hat{Q}$ reject, but $\hat{E}$ does not yet reject.
Monitoring continues with $\hat{E}$ only until 29 Mar 2019 when the first phase of monitoring ends as all three monitoring procedures have rejected the null hypothesis.
The monitoring procedure is then restarted with the $20$ observations from the time of rejection (29 Mar 2016 to 25 Apr 2016) as the stable segment and monitoring continues from 26 Apr 2016 until 23 Jun 2016 (day of the UK EU referendum), when $\hat{E}$ and $\hat{P}$ reject.
After these rejections, monitoring continues for one more day, until 24 Jun 2016, when also $\hat{Q}$ rejects.
Finally, the monitoring procedure is restarted with the next 20 observations (24 Jun 2016 to 21 Jul 2016) and monitoring continues until 31 Dec 2019 without rejections by any of the three detectors.
In this example, we see that the three detectors behave quite similar, as each of them rejects twice around the time of the UK EU referendum and no further changes afterwards.
\medskip 

\begin{figure}
\vspace{-1cm}
\begin{center}
\includegraphics[width = \linewidth]{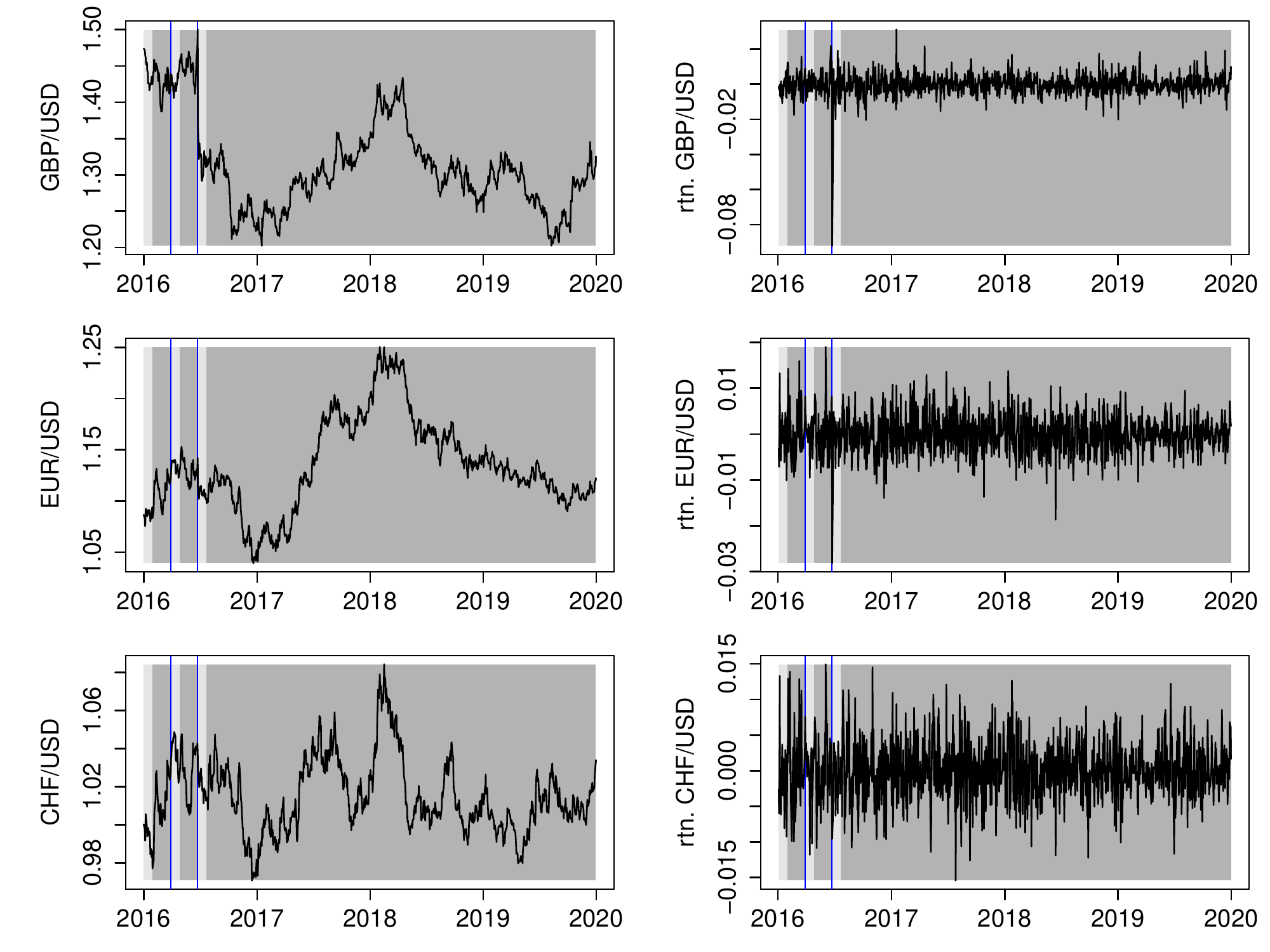}
\end{center}
\caption{\it Results of the analysis detailed in Section~\ref{sec:RealData} applied to log returns of three foreign exchange rates.
Response: GBP/USD (top), predictors: EUR/USD (middle) and CHF/USD (bottom).
Paramters used were $m := 20$, $\alpha = 0.05$ and $\gamma = 0$.
Shaded areas indicate observations that were used as the stable segment (light gray) or monitored for changes with $\hat{E}$ (dark gray).
Vertical blue lines indicate times the sequential procedure $\hat{E}$ rejected.
The times the sequential procedures $\hat{E}$/$\hat{P}$/$\hat{Q}$ stopped were: 29-03-2016/17-03-2016/17-03-2016 and 23-06-2016/23-06-2016/24-06-2016.
\label{fig:EE2}}
\end{figure}

As our second example, we consider the relation of the UK's market to that of the United States (US) and the EU.
More precisely, we consider daily log returns of the FTSE~100, a share index of the 100 companies listed on the London Stock Exchange with the highest market capitalization, as a response $Y_t$ of the linear model described in~\eqref{eqn:LMempEx}.
As predictors we consider the log returns of two similarly constructed indices that are related to the US and EU markets, namely the S\&P~500 ($P_{1,t}$) and the EuroStoxx~50 ($P_{2,t}$).
A graphical representation of the prices and log returns for the period from January 2016 to December 2019 can be seen in Figure~\ref{fig:EE1}.
The outcomes of the previously described analysis are presented visually in the graphs.
The first 20 observations (6 Jan 2016 to 1 Feb 2016) were used as the stable segment for the first phase of monitoring.
The monitoring starts on 2 Feb 2016 and on with all three detectors until 6 Feb 2017 when $\hat{E}$ rejects, but $\hat{P}$ and $\hat{Q}$ do not yet reject.
Monitoring continues with only $\hat{P}$ and $\hat{Q}$ until 16 Mar 2017 when the first phase of monitoring ends with $\hat{P}$ and $\hat{Q}$ also rejecting.
In Figure~\ref{fig:EE1}, the time that was only monitored by $\hat{P}$ and $\hat{Q}$ is not shaded in gray, because $\hat{E}$ has already rejected.

For the second phase of monitoring the procedures are then restarted with the 20 observations from the time of rejection (16 Mar 2016 to 7 Apr 2016) as the stable segment and monitoring continues from 10 Apr 2016 until 18 Apr 2017 when $\hat{P}$ rejects.
Monitoring continues with only $\hat{E}$ and $\hat{Q}$ until 24 Apr 2017 when the second phase of monitoring ends with $\hat{P}$ and $\hat{Q}$ both also rejecting.

For the third phase of monitoring the procedures are then restarted again with the 20 observations from the time of rejection (24 Apr 2017 to 18 May 2017) as the stable segment and monitoring continues from 19 May 2017 until 4 Dec 2018 when $\hat{E}$ rejects.
Monitoring continues with $\hat{P}$ and $\hat{Q}$ only until 14 Aug 2019 when the third phase of monitoring ends with $\hat{P}$ and $\hat{Q}$ both also rejecting.

For the fourth and final phase of monitoring the procedures are then restarted again with the 20 observations from the time of rejection (14 Aug 2019 to 6 Sep 2019) as the stable segment and monitoring continues from 9 Sep 2019 until 11 Oct 2019 when $\hat{E}$ and $\hat{P}$ both reject.
Monitoring continues with $\hat{Q}$ only until 31 Dec 2019, the end of the available data, without a rejection of $\hat{Q}$.
In this example, we see that $\hat{E}$, as expected from the simulations, is capable of detecting changes earlier after a longer period of monitoring.
Only in the second period, where the rejection happens early, this is not the case.

\begin{figure}
\vspace{-1cm}
\begin{center}
\includegraphics[width = \linewidth]{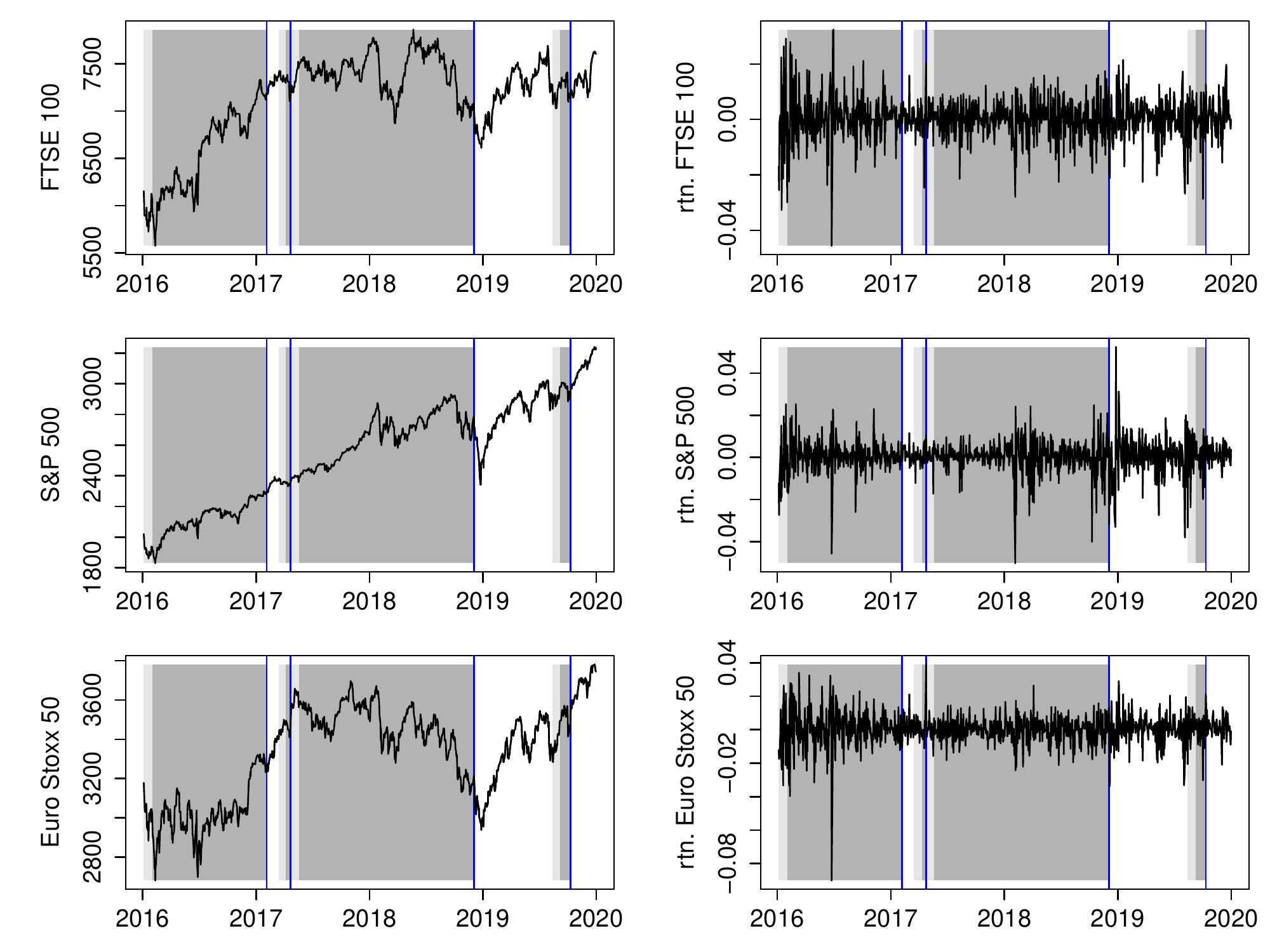}
\end{center}
\caption{\it Results of the analysis detailed in Section~\ref{sec:RealData} applied to log returns of three market indices.
Response: FTSE~100 (top), predictors: S\&P~500 (middle) and EuroStoxx~50 (bottom).
Parameters used were $m := 30$, $\alpha = 0.05$ and $\gamma = 0$.
Shaded areas indicate observations that were used as the stable segment (light gray) or monitored for changes with $\hat{E}$ (dark gray).
Vertical blue lines indicate times the sequential procedure $\hat{E}$ rejected.
The times the sequential procedures $\hat{E}$/$\hat{P}$/$\hat{Q}$ stopped were: 06-02-2016/16-03-2016/16-03-2016, 24-04-2017/18-04-2017/24-04-2017, 04-12-2018/14-08-2019/14-08-2019 and 11-10-2019/11-10-2019/didn't stop.
\label{fig:EE1}}
\end{figure}

\section{Conclusion and outlook}\label{sec:ConAndOut}
In this paper we developed a new monitoring scheme for change point detection  in a parameter of multivariate time series which is applicable in an open-end scenario.
Compared to the commonly used methods we replace the  estimator of the parameter 
from the initial sample $X_{1}, \ldots , X_{m}$ 
by an  estimator  from the sample $X_{1}, \ldots , X_{m+j}$. We then compare this estimator  with the estimator from the sample $X_{m+j+ 1}, \ldots , X_{m+k}$
for every $j=0,\ldots , k-1$,
For the new statistic the asymptotic distribution under the null hypothesis and the consistency of a corresponding test, which controls the type I error, are  established.
By considering a common class of weight functions $w_\gamma$ defined in \eqref{eq:threshold} the limit reduces to an elementary distribution, for which quantiles can be obtained by straightforward Monte Carlo simulations.
Finally, we demonstrate by a comprehensive simulation study that the new monitoring scheme is superior (in terms of testing power) to a benchmark consisting of common methods proposed in the literature.
The new statistic can also be used in closed-end scenarios, for which the same superiority in power is observed.\\
For a future research project it is of interest to replace Assumption \ref{assump:approx} by
an FCLT for any fixed time horizon and H\'{a}y\'{e}k-R\'{e}yni-inequalities, as done for instance in \cite{Kirch2018} and \cite{Kirch2019}.
Since these conditions are slightly weaker, it would be a benefit to establish the results at hand under those conditions.

Another issue - particularly with regard to our simulation study in Section \ref{sec4} - is that the test level approximation depends sensitively on the efficient estimation of the long-run variance.
The standard approach in our field, which we also followed, is to employ only the initial set for this estimate.
As the performance of this is poorly for stronger dependent models, it is logical to take a permanently updated estimate into consideration, which fits to the basic message of this work to enhance initial estimates during monitoring.
Moreover, one could tackle this problem developing a concept of self-normalization [see \cite{Shao2010}], which is applicable in an open-end scenario. 
However, as the discussion of both ideas is technically involved, it is beyond the scope of this paper and left as a promising subject for future research.

Finally, it is a logical next step to also characterize the asymptotic distribution for the stopping times based on the statistic $\hatE$ defined in \eqref{eq:hatE}.
Corresponding results are already known for the methods based on $\hatQ$ and $\hatP$, see \cite{Aue2004} and \cite{Fremdt2014}, respectively.
\medskip

\noindent \textbf{Acknowledgments}
This work has been supported in part by the Collaborative Research Center ``Statistical modeling of nonlinear dynamic processes'' (SFB 823, Teilprojekt A1, C1) and the Research Training Group 'High-dimensional phenomena in probability - fluctuations and discontinuity' (RTG 2131) of the German Research Foundation (DFG).
Moreover, the authors would like to thank Christina St\"ohr and Herold Dehling for extremely helpful discussions.
We are also grateful to the two unknown referees for their constructive comments on an earlier version of the paper and to Ivan Kojadinovic for pointing out some errors.

\setlength{\bibsep}{1pt}
\begin{small}
\bibliography{literature}
The data that support the findings of this study were obtained from the domain: \url{https://www.ariva.de}~.
\end{small}

\newpage

\appendix
\title{{Online appendix to: A new approach for open-end sequential change point monitoring}}
\setcounter{page}{1}
{
\author{
  {\small Josua G\"osmann}\\
  {\small Ruhr-Universit\"at Bochum }\\
  {\small Fakult\"at f\"ur Mathematik}\\
  {\small 44780 Bochum, Germany} \\
  {\small \href{mailto:josua.goesmann@ruhr-uni-bochum.de}{josua.goesmann@ruhr-uni-bochum.de}}\\
  {\small (corresponding author)}
  \and
  {\small Tobias Kley}\\
  {\small University of Bristol}\\
  {\small School of Mathematics}\\
  {\small Bristol BS8 1UG, United Kingdom}\\
  {\small \href{mailto:tobias.kley@bristol.ac.uk}{tobias.kley@bristol.ac.uk}}
\and
  {\small Holger Dette}\\
  {\small Ruhr-Universit\"at Bochum }\\
  {\small Fakult\"at f\"ur Mathematik}\\
  {\small 44780 Bochum, Germany} \\
  {\small \href{mailto:holger.dette@ruhr-uni-bochum.de}{holger.dette@ruhr-uni-bochum.de}}
}
\maketitle

\section{Proofs of the results in  Section \refb{sec2}}

\label{app:technical}

\noindent
\begin{proof}[\textbf{Proof of Theorem \refb{thm:mainH0}:}]

In the proof we use the following extra notation.
Define the statistic
\begin{align}\label{eq:E}
E_m(k) = m^{-1/2}\max_{j=0}^{k-1} (k - j) \Big\Vert \hattheta_{1}^{m+j} - \hattheta_{m+j+1}^{m+k} \Big\Vert_{\Sigma^{-1}}~,
\end{align}
where we have replaced the long-run variance estimator $\hatSigma$ by the (unknown) true long-run variance $\Sigma$ in definition \eqrefb{eq:hatE}.
Further define
\begin{align}\label{eq:tildeE}
\tilde{E}_m(k) = m^{-1/2} \max_{j=0}^{k-1} (k-j) \bigg\Vert \dfrac{1}{m+j}\sum_{t=1}^{m+j}\IF_t - \dfrac{1}{k-j}\sum_{t=m+j+1}^{m+k} \IF_t \bigg\Vert_{\Sigma^{-1}}~,
\end{align}
where we have replaced the estimator $\hattheta_{i}^{j}$ by corresponding averages of the influence function in \eqref{eq:E}.
Throughout the proof we will frequently use that due to continuity of $\tilde{w}$ and $\limsup_{t\to \infty} t\tilde{w}(t) < \infty$ in Assumption \refb{assump:weighting}, the weight function $w$ has a uniform upper bound, say $u_w$.
Finally, the triple $(\Omega, \mathcal{A}, \Pb)$ will denote the underlying probability space.\\

The proof itself is now split into several Lemmas \ref{lem:RemoveRemainders} - \ref{lem:simplylimit}.
The first Lemma shows that $\tilde{E}_m(k)$ and $E_m(k)$ are (asymptotically) equivalent.
Lemma \ref{lem:approx} will approximate $E_m(k)$ by Brownian motions, while Lemma \ref{lem:ObtLimit} then yields a limit for this approximation.
Lemma \ref{lem:plugcovar} finishes the proof by plugging in the covariance estimator, meaning that $\hat{E}_m(k)$ and $E_m(k)$ are asymptotically equivalent.
Finally, Lemma \ref{lem:simplylimit} will establish the other representation of the limit distribution in \eqrefb{eq:ThmMainH01}.
In each Lemma, we suppose that the assumptions of Theorem \refb{thm:mainH0} are valid.

\begin{lemma}[Remove Remainders]\label{lem:RemoveRemainders}
It holds that
\begin{align*}
\supkinf w(k/m) \Big| E_m(k) - \tilde{E}_m(k) \Big| = \op(1)
\end{align*}
as $m \to \infty$.
\end{lemma}
\begin{proof}
By the (reverse) triangle inequality and the linearization in \eqrefb{eq:meta-remainder} we obtain
\begin{align*}
&\;\Big| E_m(k) - \tilde{E}_m(k) \Big|\\
&\leq m^{-1/2} \max_{j=0}^{k-1} (k - j)\Big\Vert \hattheta_{1}^{m+j} - \hattheta_{m+j+1}^{m+k} - \dfrac{1}{m+j}\sum_{t=1}^{m+j}\IF_t + \dfrac{1}{k-j}\sum_{t=m+j+1}^{m+k} \IF_t \Big\Vert_{\Sigma^{-1}}\\
&= m^{-1/2}\max_{j=0}^{k-1} (k-j) \Big\Vert R_{1,m+j} - R_{m+j+1,m+k} \Big\Vert_{\Sigma^{-1}}~,
\end{align*}
where we used that $\theta_t$ is constant for the last equality.
Next, we obtain that
\begin{align}\label{ineq:remain1}
\begin{split}
&\supkinf \dfrac{w(k/m)}{m^{1/2}} \max_{j=0}^{k-1} (k-j) \Big\Vert R_{m+j+1,m+k} \Big\Vert_{\Sigma^{-1}}\\
= &\supkinf \dfrac{w(k/m)(m+k)^{1/2}}{m^{1/2}} \max_{j=0}^{k-1} \dfrac{k-j}{(m+k)^{1/2}} \Big\Vert R_{m+j+1,m+k} \Big\Vert_{\Sigma^{-1}}\\
\leq &\supkinf \dfrac{w(k/m)(m+k)^{1/2}}{m^{1/2}} \supkinf \max_{1 \leq i < j \leq m+k} \dfrac{j-i+1}{(m+k)^{1/2}} \Big\Vert R_{i,j} \Big\Vert_{\Sigma^{-1}}~.
\end{split}
\end{align}
Similar as in \eqref{ineq:remain1} it holds
\begin{align}\label{ineq:remain2}
\begin{split}
&\supkinf \dfrac{w(k/m)}{m^{1/2}} \max_{j=0}^{k-1} (k-j) \Big\Vert R_{1,m+j} \Big\Vert_{\Sigma^{-1}}\\
\leq &\supkinf \dfrac{w(k/m)(m+k)}{m} \sqrt{m}\supkinf \max_{j=0}^{k-1} \Big\Vert R_{1,m+j} \Big\Vert_{\Sigma^{-1}}\\
= &\supkinf \dfrac{w(k/m)(m+k)}{m} \sup_{k=0}^{\infty} \sqrt{m} \Big\Vert R_{1,m+k} \Big\Vert_{\Sigma^{-1}}
\end{split}
\end{align}
Using Assumption \refb{assump:weighting} for the weight function $w$ we obtain
\begin{align}\label{eq:threshold-remainder}
\supkinf \dfrac{w(k/m)(m+k)}{m} 
= \supkinf w(k/m)\Big(1+\dfrac{k}{m}\Big)
\leq \sup_{t > 0} w(t)(1+t)< \infty~.
\end{align}
and by similar arguments it holds that
\begin{align}\label{eq:threshold-remainder2}
\supkinf \dfrac{(m+k)^{1/2}w(k/m)}{\sqrt{m}} < \infty~. 
\end{align}
Further note that due to Assumption \refb{assump:remainder} with probability one
\begin{align}\label{eq:threshold-remainder3}
\begin{split}
\supkinf \sqrt{m} \Big| R_{1,m+k} \Big|
\leq \supkinf \sqrt{m+k} &\Big| R_{1,m+k} \Big|
= \sup_{k=m}^\infty \sqrt{k} \Big| R_{1,k} \Big|\\
&\overset{j=k,i=1}{\leq} \sup_{k=m}^\infty \max_{1\leq i < j \leq k}\dfrac{j-i+1}{\sqrt{k}} \Big| R_{i,j} \Big|
= o(1)~.
\end{split}
\end{align}
Now combining \eqref{eq:threshold-remainder}, \eqref{eq:threshold-remainder2} and \eqref{eq:threshold-remainder3} the bounds derived in \eqref{ineq:remain1} and \eqref{ineq:remain2} are of order $\op(1)$, which finishes the proof of Lemma \ref{lem:RemoveRemainders}.
\end{proof}

\noindent For the proof of the next Lemma we can proceed (roughly) similar to the proof of Lemma 5.2 in \citesuppl{Fremdt2015}.

\begin{lemma}[Approximation with Brownian motions]\label{lem:approx}
Define 
\begin{align*}
P_m(k)
= \dfrac{1}{\sqrt{m}} \max_{j=0}^{k-1}
\Big| W_{m,1}(k) - \dfrac{m+k}{m+j} W_{m,1}(j) + \dfrac{k-j}{m+j}W_{m,2}(m)\Big|~,
\end{align*}
then 
\begin{align*}
\supkinf w(k/m) \Big| \tilde{E}_m(k) - P_m(k) \Big|
= \op(1)
\end{align*}
as $m \to \infty$.
\end{lemma}
\begin{proof}
For the remainder of the proof let $W_{m,i}^{\Sigma}:= \sqrt{\Sigma}W_{m,i}$ for $i=1,2$ and note that this implies
\begin{align*}
P_m(k) 
= \dfrac{1}{\sqrt{m}} \max_{j=0}^{k-1} 
\Big\Vert \WmSig(k) - \dfrac{m+k}{m+j} \WmSig(j) + \dfrac{k-j}{m+j}\WmSigII(m)\Big\Vert_{\Sigma^{-1}}~.
\end{align*}
The last display and the (reverse) triangle inequality then yield
\begin{align}\label{ineq:approxsep}
\begin{split}
&\;\;\Big| \tilde{E}_m(k) - P_m(k) \Big| \\
&\leq \dfrac{1}{\sqrt{m}}\max_{j=0}^{k-1} \Big\Vert \dfrac{k-j}{m+j} \sum_{t=1}^{m+j}\IF_t - \sum_{t=m+j+1}^{m+k} \IF_t\\
&\qquad +\WmSig(k) - \dfrac{m+k}{m+j} \WmSig(j) + \dfrac{k-j}{m+j}\WmSigII(m) \Big\Vert_{\Sigma^{-1}}\\
&\leq \dfrac{1}{\sqrt{m}} \Big\Vert \sum_{t=m+1}^{m+k} \IF_t - \WmSig(k)  \Big\Vert_{\Sigma^{-1}}
+ \dfrac{1}{\sqrt{m}}\max_{j=0}^{k-1} \dfrac{m+k}{m+j}\Big\Vert \sum_{t=m+1}^{m+j} \IF_t - \WmSig(j) \Big\Vert_{\Sigma^{-1}}\\
&\qquad + \dfrac{1}{\sqrt{m}}\max_{j=0}^{k-1}\dfrac{k-j}{m+j} \Big\Vert \sum_{t=1}^{m}\IF_t - \WmSigII(m) \Big\Vert_{\Sigma^{-1}}~.
\end{split}
\end{align}
We will treat the three summands of the last display separately.
Using the definition of the operator norm, we derive the following bound for the first summand
\begin{align*}\
&\supkinf \dfrac{w(k/m)}{\sqrt{m}} \Big\Vert \sum_{t=m+1}^{m+k}\IF_t - \WmSig(k) \Big\Vert_{\Sigma^{-1}}\\
\leq &\supkinf \dfrac{w(k/m)k^\xi}{\sqrt{m}} 
\supkinf \dfrac{1}{k^\xi}\Big\Vert \sum_{t=m+1}^{m+k}\IF_t - \WmSig(k) \Big\Vert_{\Sigma^{-1}}
\end{align*}
By Assumption \refb{assump:approx} and the estimate $\Vert x \Vert_{\Sigma^{-1}} \leq \Vert \Sigma^{-1} \Vert_{op}^{1/2}|x|$ for all $x \in \R^p$, the second factor is of order $\Op(1)$.
Since $w$ has an upper bound $u_w$ we obtain for the first factor, that
\begin{align*}
\sup_{k=1}^{m} &\dfrac{w(k/m)k^\xi}{\sqrt{m}} 
\leq \sup_{k=1}^{m} \dfrac{u_wk^\xi}{\sqrt{m}} 
= u_w m^{\xi-1/2} = o(1)~,\\
\sup_{k=m+1}^{\infty} &\dfrac{w(k/m) k^\xi}{\sqrt{m}} 
\leq \sup_{k=m+1}^{\infty} \dfrac{1}{k^{1/2-\xi}}
\supkinf \dfrac{w(k/m)k^{1/2}}{\sqrt{m}}
\leq \dfrac{1}{m^{1/2-\xi}}\sup_{t> 0} \sqrt{t}w(t) = o(1)~.
\end{align*}
Next, we can bound the second summand on the right-hand side in \eqref{ineq:approxsep} by
\begin{align*}
&\supkinf \dfrac{w(k/m)}{\sqrt{m}} \max_{j=0}^{k-1} \dfrac{m+k}{m+j}\Big\Vert \sum_{t=m+1}^{m+j} \IF_t - \WmSig(j) \Big\Vert_{\Sigma^{-1}}\\
\leq &\supkinf \dfrac{w(k/m)(m+k)}{\sqrt{m}} \max_{j=1}^{k-1} \dfrac{1}{j^\xi m^{1-\xi}}\Big\Vert \sum_{t=m+1}^{m+j} \IF_t - \WmSig(j) \Big\Vert_{\Sigma^{-1}}\\
\leq &\supkinf \dfrac{w(k/m)(m+k)m^{\xi-1}}{\sqrt{m}} \supkinf\dfrac{1}{k^\xi}\Big\Vert \sum_{t=m+1}^{m+k} \IF_t - \WmSig(k) \Big\Vert_{\Sigma^{-1}}~,
\end{align*}
where we used that $(m+j) = (m+j)^\xi(m+j)^{1-\xi} \geq j^\xi m^{1-\xi}$.
Using again Assumption \refb{assump:approx}, the second factor in the last display is of order $\Op(1)$.
Moreover, following the idea of the proof of Lemma 3 in \citesuppl{Aue2006} it holds that
\begin{align}\label{eq:thsholdmk}
\begin{split}
\sup_{k=1}^m \dfrac{w(k/m)(m+k)m^{\xi-1}}{\sqrt{m}}
&\leq u_w\sup_{k=1}^m (m+k)m^{\xi-3/2}
= u_w 2m^{\xi-1/2}
= o(1)~,\\
\sup_{k=m+1}^\infty \dfrac{w(k/m)(m+k)m^{\xi-1}}{\sqrt{m}}
&= m^{\xi-1/2}\sup_{k=m+1}^\infty \big(1+k/m\big)w(k/m)\\
&\hspace{2cm}\leq m^{\xi-1/2}\sup_{t>1} (1+t)w(t) = o(1)
\end{split}
\end{align}
and it remains to treat the third summand of the right-hand side in \eqref{ineq:approxsep}, which can be bounded by
\begin{align*}
&\supkinf \dfrac{w(k/m)}{\sqrt{m}}\max_{j=0}^{k-1}\dfrac{k-j}{m+j} \Big\Vert \sum_{t=1}^{m}\IF_t - \WmSigII(m) \Big\Vert_{\Sigma^{-1}}\\
\leq &\supkinf \dfrac{w(k/m)}{\sqrt{m}} \dfrac{k}{m} \Big\Vert \sum_{t=1}^{m}\IF_t - \WmSigII(m) \Big\Vert_{\Sigma^{-1}}\\
\leq &\supkinf \dfrac{w(k/m)km^{\xi-1}}{\sqrt{m}} \dfrac{1}{m^\xi} \Big\Vert \sum_{t=1}^{m}\IF_t - \WmSigII(m) \Big\Vert_{\Sigma^{-1}}~.
\end{align*}
Using Assumption \refb{assump:approx} and the arguments in \eqref{eq:thsholdmk} this term is of order $\op(1)$, which finishes the proof of Lemma \ref{lem:approx}.
\end{proof}

\begin{lemma}[Obtain limit process]\label{lem:ObtLimit}
The following weak convergence holds
\begin{align*}
\supkinf w(k/m) P_m(k)
\convd \sup_{0 \leq t < \infty} \max_{0 \leq s \leq t} 
w(t) \bigg| W_1(t)- \dfrac{1+t}{1+s}W_1(s) + \dfrac{t-s}{1+s} W_2(1)  \bigg|
\end{align*}
as $m \to \infty$, where $W_1$ and $W_2$ denote independent, $p$-dimensional standard Brownian motions.
\end{lemma}
\begin{proof}
First note that due to the scaling properties of the Brownian motion [see for example  page 30 of \citesuppl{Shorack2009}], it holds in distribution that
\begin{align}\label{eq:P2tildeP}
P_m(k)
&\eqd \max_{j=0}^{k-1} \Big| W_1(k/m) - \dfrac{1+k/m}{1+j/m}W_1(j/m) + \dfrac{k/m-j/m}{1+j/m}W_2(1)\Big|
:= \widetilde{P}_m(k)
\end{align}
and so within the proof we will - without loss of generality - only consider $\widetilde{P}_m(k)$.
Additionally, we define the processes
\begin{align*}
L^{(1)}(s,t) &:= \tildew(t) \Big( \dfrac{t+1}{s+1}W_1(s) - W_1(t) \Big)~,\\
L^{(2)}(s,t) &:= \tildew(t) \dfrac{t-s}{s+1}W_2(1)~.
\end{align*}
We will show that $L^{(i)}(s,t)$ for $i=1,2$ are uniformly continuous on $\R_\Delta^+ = \{ (s,t) \in \R^2 \;|\; 0\leq s \leq t \leq T_w\}$ with probability one, where $T_w \in \R_{>0} \cup \{\infty\}$ is the (right) cutoff constant from Assumption \refb{assump:weighting}.
In case $T_w < \infty$ this directly follows as both processes are already a.s. continuous, so only $T_w =\infty$ is of interest.\smallskip\\
\textbf{Case $L^{(1)}$}:
In the following let $\varepsilon>0$ be fixed but arbitrary.
Next, we fix one $\omega_0 \in \Omega$, such that  $W_1$ fulfills the law of iterated logarithm  and is continuous.
As this event has probability one, it suffices to show that $L^{(1)}$ is uniformly continuous for $W_1=W_1(\cdot,\omega_0)$.
For the ease of reading we will omit $\omega_0$ in the presentation below.
By the law of iterated logarithm, there exist $C=C(\varepsilon,\omega_0)$ sufficiently large, such that
\begin{align}\label{ineq:ChoiceC}
\sup_{t\geq C} \dfrac{|W_1(t)|}{t} < \dfrac{\varepsilon}{4B}~,
\end{align}
with $B$ chosen as
\begin{align*}
B:= \sup_{t >0}\; (t+2)\tildew(t) <\infty~.
\end{align*}
\\
Depending on $C$, we can split $\R_\Delta^+$ into the (overlapping) sets
\begin{align}\label{eq:decompDelta}
\R_\Delta^+ = \mathcal{M}_1(C) \cup \mathcal{M}_2(C) \cup \mathcal{M}_3(C)~,
\end{align}
where we use the definitions
\begin{align*}
\mathcal{M}_1(C)
&= \R_\Delta^+ \cap [0,C+1]^2 ~,\\
\mathcal{M}_2(C)
&= \R_\Delta^+ \cap [0,C+1] \times [C,\infty) ~,\\
\mathcal{M}_3(C)
&= \R_\Delta^+ \cap [C,\infty)^2 ~.
\end{align*}
Further let $d$ denote the maximum distance that is
\begin{align*}
d\big((s_1,t_1),(s_2,t_2)\big)
= \max\big\{|s_1-s_2|,|t_1-t_2| \big\}~.
\end{align*}
Note that by construction of the decomposition in \eqref{eq:decompDelta}, whenever $d((s_1,t_1),(s_2,t_2))<\delta$ for sufficiently small $\delta>0$, then there is $j \in {1,2,3}$, such that both pairs are in the same subset $\mathcal{M}_j(C)$.
Thus the uniform continuity of $L^{(1)}$ follows if we can choose $\delta>0$ sufficiently small, such that
\begin{align}\label{ineq:forjs}
\sup_{\substack{d((s_1,t_1),(s_2,t_2))<\delta\\(s_1,t_1),(s_2,t_2) \in \mathcal{M}_j(C)}}
\Big| L^{(1)}(s_1,t_1) - L^{(1)}(s_2,t_2) \Big|
< \varepsilon
\end{align}
for $j=1,2,3$.
In the following, we will treat each subset separately.\smallskip\\
\textbf{Set $\mathcal{M}_1(C)$}:
As this set is compact, the (ordinary) almost sure continuity of $L^{(1)}$ already implies that \eqref{ineq:forjs} holds for $j=1$ and $\delta>0$ sufficiently small.\smallskip\\
\textbf{Set $\mathcal{M}_2(C)$}:
We have the following bound
\begin{align}\label{ineq:M2C1}
\begin{split}
&\sup_{\substack{d((s_1,t_1),(s_2,t_2))<\delta\\(s_1,t_1),(s_2,t_2) \in \mathcal{M}_2(C)}}
\Big| L^{(1)}(s_1,t_1) - L^{(1)}(s_2,t_2) \Big|\\
&\leq \sup_{\substack{|t_1-t_2|<\delta\\0 \leq s \leq C+1,\;\; t_1,t_2> C}}
\Big| L^{(1)}(s,t_1) - L^{(1)}(s,t_2) \Big| +
\sup_{\substack{|s_1-s_2|<\delta\\0 \leq s_1,s_2 \leq C+1,\;\; t> C}}
\Big| L^{(1)}(s_1,t) - L^{(1)}(s_2,t) \Big|~.
\end{split}
\end{align}
We will treat both summands of the last display individually.
The first one can be bounded again as follows
\begin{align}\label{ineq:M2C2}
\begin{split}
&\sup_{\substack{|t_1-t_2|<\delta\\0 \leq s \leq C+1,\;\; t_1,t_2> C}}
\Big| L^{(1)}(s,t_1) - L^{(1)}(s,t_2) \Big|\\
&\leq \sup_{\substack{|t_1-t_2|<\delta\\0 \leq s \leq C+1,\;\; t_1,t_2> C}} 
\Big|\tildew(t_1)- \tildew(t_2) \Big|\Big| \dfrac{t_1+1}{s+1}W_1(s) - W_1(t_1) \Big|\\
&\qquad\qquad\qquad\qquad\qquad+ \tildew(t_2)\Big| \dfrac{t_1-t_2}{s+1}W_1(s) - W_1(t_1)+W_1(t_2) \Big|
\end{split}
\end{align}
and again we will treat both terms separately.
For the first term note the upper bound
\begin{align*}
&\sup_{\substack{|t_1-t_2|<\delta\\0 \leq s \leq C+1,\;\; t_1,t_2> C}} 
\big|\tildew(t_2)-\tildew(t_1)\big|(t_1+1)\Big| \dfrac{W_1(s)}{s+1} - \dfrac{W_1(t_1)}{t_1+1} \Big|\\
&\qquad\leq 2 u_w\sup_{\substack{|t_1-t_2|<\delta\\}} \bigg|\dfrac{1}{\tildew(t_2)} - \dfrac{1}{\tildew(t_1)}\bigg|\sup_{t\geq C}(t+1)\tildew(t)\sup_{t \geq 0}\dfrac{|W_1(t)|}{t+1}~,
\end{align*}
and since $1/\tildew$ is uniformly continuous this expression is smaller than $\varepsilon/3$ for sufficiently small $\delta$.
For the second term of the right-hand side of \eqref{ineq:M2C2}, note that we have the bound
\begin{align*}
&\sup_{\substack{|t_1-t_2|<\delta\\0 \leq s \leq C+1,\;\; t_1,t_2> C}} 
\tildew(t_2)\Big| \dfrac{t_1-t_2}{s+1}W_1(s) - W_1(t_1) + W_1(t_2) \Big|\\
\leq &\sup_{\substack{|t_1-t_2|<\delta\\0 \leq s \leq C+1,\;\; t_1,t_2> C}}
\tildew(t_2)|t_1-t_2|\dfrac{|W_1(s)|}{s+1} + \tildew(t_2)(t_1+1)\dfrac{|W_1(t_1)|}{t_1+1}+
\tildew(t_2)(t_2+1)\dfrac{|W_1(t_2)|}{t_2+1}\\
\leq &\delta u_w\sup_{s\geq 0} \dfrac{|W_1(s)|}{s+1} + 2\sup_{t>0} (t+2)\tildew(t)\cdot\sup_{t>C}\dfrac{|W_1(t)|}{t+1}
\end{align*}
and by the choice of $C$ in \eqref{ineq:ChoiceC} and for sufficiently small $\delta$ this is bounded by $\varepsilon/3$.
To complete the treatment of $\mathcal{M}_2(C)$ it only remains to examine the second term on the right-hand side of \eqref{ineq:M2C1}.
We obtain that
\begin{align*}
&\sup_{\substack{|s_1-s_2|<\delta\\0 \leq s_1,s_2 \leq C+1,\;\; t> C}}
\Big| L^{(1)}(s_1,t) - L^{(1)}(s_2,t) \Big|\\
&\leq \sup_{t > 0} (t+1)\tildew(t) \sup_{\substack{|s_1-s_2|<\delta\\0 \leq s_1,s_2 \leq C+1}}
\Big| \dfrac{W_1(s_1)}{s_1+1} - \dfrac{W_1(s_2)}{s_2+1} \Big|~,
\end{align*}
which can be bounded by $\varepsilon/3$ for sufficiently small $\delta$ since the first factor is a constant and the function $f(s)=W_1(s+1)/(s+1)$ is uniformly continuous on the compact set $[0,C+1]$.\smallskip\\
\textbf{Set $\mathcal{M}_3(C)$}:
Note that
\begin{align*}
\sup_{\substack{d((s_1,t_1),(s_2,t_2))<\delta\\(s_1,t_1),(s_2,t_2) \in \mathcal{M}_3(C)}}
\Big| L^{(1)}(s_1,t_1) - L^{(1)}(s_2,t_2) \Big|
&\leq 2\sup_{s,t>C} |L^{(1)}(s,t)|\\
&\leq 2\sup_{t > 0} (t+1)\tildew(t) \sup_{s,t>C} \Big| \dfrac{W_1(s)}{s+1} - \dfrac{W_1(t)}{t+1} \Big|\\
&\leq 4 \sup_{t > 0} (t+1)\tildew(t) \sup_{t\geq C} \dfrac{|W_1(t)|}{t}
< \varepsilon~,
\end{align*}
where we used the choice of $C$ in \eqref{ineq:ChoiceC} for the last estimate. \\
This completes the third case and so the almost sure uniform continuity of $L^{(1)}$ on the set $\R_\Delta^+$ is established.
\smallskip\\

\noindent \textbf{Case $L^{(2)}$}:
Again let $\varepsilon>0$ and suppose that $d\big((s_1,t_1),(s_2,t_2)\big)<\delta$.
It holds that
\begin{align}\label{ineq:boundL1}
\big|L^{(2)}(s_1,t_1) - L^{(2)}(s_2,t_2)\big|
\leq \big|L^{(2)}(s_1,t_1) - L^{(2)}(s_2,t_1)\big| + \big|L^{(2)}(s_2,t_1) - L^{(2)}(s_2,t_2) \big|
\end{align}
and note for the first summand of the last display that
\begin{align*}
\big|L^{(2)}(s_1,t_1) - L^{(2)}(s_2,t_1)\big|
&= |W_2(1)|\tildew(t_1)\bigg| \dfrac{t_1-s_1}{s_1+1} - \dfrac{t_1-s_2}{s_2+1} \bigg|\\
&= |W_2(1)|(t_1+1)\tildew(t_1)\bigg| \dfrac{s_2-s_1}{(s_1+1)(s_2+1)}\bigg|\\
&\leq \big|W_2(1)\big|(t_1+1)\tildew(t_1)\big|s_2-s_1\big|
\end{align*}
and by Assumption \refb{assump:weighting} the last term is smaller than $\varepsilon/2$ uniformly for all $t_1>0$ if $\delta >0$ is chosen sufficiently small.
It remains to examine the second summand of the right-hand side of \eqref{ineq:boundL1}.
It holds that
\begin{align*}
|L^{(2)}(s_2,t_1) - L^{(2)}(s_2,t_2) \big|
&\leq \big|\tildew(t_1) - \tildew(t_2) \big| \dfrac{|t_1-s_2|}{s_2+1}|W_2(1)| +
\tildew(t_2)|W_2(1)|\dfrac{|t_2-t_1|}{s_2+1}\\
&\leq \bigg|\dfrac{1}{\tildew(t_1)} - \dfrac{1}{\tildew(t_2)}\bigg|\tildew(t_1) \tildew(t_2) \dfrac{t_1+\delta}{s_2+1}|W_2(1)| +
u_w|W_2(1)||t_2-t_1|\\
&\leq u_w\bigg|\dfrac{1}{\tildew(t_1)} - \dfrac{1}{\tildew(t_2)}\bigg|\tildew(t_1)(t_1+\delta)|W_2(1)| +
u_w|W_2(1)||t_2-t_1|~,
\end{align*}
where we used that $s_2\leq t_2 \leq t_1 + \delta$ whenever $d\big((s_1,t_1),(s_2,t_2)) < \delta$.
By Assumption \refb{assump:weighting} the last display is smaller than $\varepsilon/2$ whenever $\delta>0$ is sufficiently small and so the almost sure continuity of $L^{(2)}$ is shown.\medskip\\
Finally, we can combine our observations to finish the proof.
Note that by the results above, also the process $L(s,t) :=| L^{(1)}(s,t) + L^{(2)}(s,t)|$ is uniformly continuous with probability one.
Next, recall the cutoff parameters from Assumption \refb{assump:weighting} and observe the identity
\begin{align*}
\sup_{t_w m \leq t < T_w m} \sup_{0 \leq s \leq t} L\big(\floor{s}/m, \floor{t}/m\big)
= \supkinf w(k/m)\widetilde{P}_m(k)
= \sup_{k=t_w \cdot m}^{T_w m} \tilde{w}(k/m)\widetilde{P}_m(k)~.
\end{align*}
Furthermore, note that
\begin{align*}
\sup_{1 \leq t < \infty} \sup_{0 \leq s \leq t} L\big(\floor{s}/m, \floor{t}/m\big) 
= \sup_{0 \leq t < \infty} \sup_{0 \leq s \leq t} L\big(\floor{s}/m, \floor{t}/m\big) + o(1)
\end{align*}
almost surely as $m \to \infty$.
Now we can finish the proof of Lemma \ref{lem:ObtLimit} using the almost sure uniform continuity of $L$, which implies that for arbitrary $\varepsilon>0$ and almost every $\omega \in \Omega$ we can choose sufficiently large $m=m(\varepsilon,\omega)$ such that
\begin{align*}
&\bigg| \supkinf w(k/m)\tilde{P}_m(k) - \sup_{0 \leq t < \infty} \max_{0 \leq s \leq t} 
w(t) \Big| W_1(t)- \dfrac{1+t}{1+s}W_1(s) + \dfrac{t-s}{1+s} W_2(1) \Big| \bigg|\\
&= \bigg| \sup_{k=t_w\cdot m}^{T_w m} \tildew(k/m)\tilde{P}_m(k) - \sup_{t_w \leq t < T_w} \max_{0 \leq s \leq t} 
 \tildew(t) \Big| W_1(t)- \dfrac{1+t}{1+s}W_1(s) + \dfrac{t-s}{1+s} W_2(1) \Big| \bigg|\\
&= \bigg| \sup_{t_w m \leq t < T_w m} \sup_{0 \leq s \leq t} L\big(\floor{s}/m, \floor{t}/m\big) -
\sup_{t_w \leq t < T_w} \sup_{0 \leq s \leq t} L\big(s, t\big) \bigg|\\
&= \bigg| \sup_{t_w m \leq t < T_w m} \sup_{0 \leq s \leq t} L\big(\floor{s}/m, \floor{t}/m\big) -
\sup_{t_w m \leq t < T_w m} \sup_{0 \leq s \leq t} L\big(s/m, t/m\big) \bigg|\\
&\leq \sup_{t_w m \leq t < T_w m} \sup_{0 \leq s \leq t} \bigg| L\big(\floor{s}/m, \floor{t}/m\big) - L\big(s/m, t/m\big)\bigg|\\
&\leq \sup_{0 \leq t < \infty} \sup_{0 \leq s \leq t} \bigg| L\big(\floor{s}/m, \floor{t}/m\big) - L\big(s/m, t/m\big)\bigg|\\
&\leq \sup_{\substack{d((s_1,t_1),(s_2,t_2))< 1/m\\(s_1,t_1),(s_2,t_2)\in \R_\Delta^+}} \bigg| L\big(s_1, t_1\big) - L\big(s_2, t_2\big)\bigg|
< \varepsilon~.
\end{align*}
\end{proof}

\noindent Combining Lemma \ref{lem:RemoveRemainders}, \ref{lem:approx} and \ref{lem:ObtLimit} we have already proven that
\begin{align}\label{conv:E}
\supkinf w(k/m)E_m(k)
\convd
\max_{0 \leq t < \infty} \max_{0 \leq s \leq t} w(t)
\Big| W_1(t)- \dfrac{1+t}{1+s} W_1(s) + \dfrac{t-s}{1+s}W_2(1) \Big|~,
\end{align}
and it only remains to investigate the impact of the covariance estimator.
Therefore the following Lemma finishes the proof of Theorem \refb{thm:mainH0}.

\begin{lemma}[Plug in of covariance estimator]\label{lem:plugcovar}
We have that
\begin{align*}
\supkinf w(k/m)E_m(k) - \supkinf w(k/m)\hatE_m(k) = \op(1)~.
\end{align*}
\end{lemma}
\begin{proof}\renewcommand{\qedsymbol}{}
Observe the bound
\begin{align}\label{ineq:covarplug}
\bigg| &\supkinf w(k/m)E_m(k) - \supkinf w(k/m)\hatE_m(k) \bigg|\notag\\
&\leq \supkinf \dfrac{w(k/m)}{\sqrt{m}}\max_{j=0}^{k-1} (k-j)\bigg| (\hattheta_{1}^{m+j} - \hattheta_{m+j+1}^{m+k})^\top(\hatSigma^{-1}-\Sigma^{-1})(\hattheta_{1}^{m+j} - \hattheta_{m+j+1}^{m+k})\bigg|^{1/2}~.
\end{align}
Next note that for a symmetric matrix $A$ and an arbitrary vector $v$ the Cauchy-Schwarz inequality implies
\begin{align*}
|v^\top A v|^2
\leq |Av||v|
\leq \Vert A\Vert_{op}|v|^2~,
\end{align*}
and we can bound \eqref{ineq:covarplug} by
\begin{align}\label{bound:covar}
\Big\Vert \hatSigma^{-1}-\Sigma^{-1} \Big\Vert_{op}^{1/2} \supkinf \dfrac{w(k/m)}{\sqrt{m}}\max_{j=0}^{k-1}(k-j)\Big| \hattheta_{1}^{m+j} - \hattheta_{m+j+1}^{m+k}\Big|~.
\end{align}
Since $\hatSigma$ is a consistent estimator of $\Sigma$, an application of the continuous mapping theorem yields
\begin{align}\label{opnorm:covar}
\Big\Vert \hatSigma^{-1}-\Sigma^{-1} \Big\Vert_{op}
= \op(1)~.
\end{align}
Next, the definition of the operator norm yields
\begin{align*}
\supkinf &\dfrac{w(k/m)}{\sqrt{m}}
\max_{j=0}^{k-1}(k-j)\big| \hattheta_{1}^{m+j} - \hattheta_{m+j+1}^{m+k} \big|\\
&\leq \Vert \Sigma^{1/2}\Vert_{op} \supkinf \dfrac{w(k/m)}{\sqrt{m}}
\max_{j=0}^{k-1}(k-j)\big\Vert \hattheta_{1}^{m+j} - \hattheta_{m+j+1}^{m+k} \big\Vert_{\Sigma^{-1}}\\
&= \Vert \Sigma^{1/2}\Vert_{op} \supkinf w(k/m)E_m(k)
= \Op(1)~.
\end{align*}
Now a combination of \eqref{conv:E} and \eqref{opnorm:covar} implies that the expression in \eqref{bound:covar} is of order $\op(1)$, which completes the proof of Lemma \ref{lem:plugcovar} and thus also the proof of Theorem \refb{thm:mainH0}.
\end{proof}
\noindent Combining Lemmas \ref{lem:RemoveRemainders}, \ref{lem:approx}, \ref{lem:ObtLimit} and \ref{lem:plugcovar} we have now established that
\begin{align*}
\supkinf w(k/m)\hatE_m(k)
&\convd
\sup_{0 \leq t < \infty} \max_{0 \leq s \leq t} w(t) \bigg| W_1(t)- \dfrac{1+t}{1+s}W_1(s) + \dfrac{t-s}{1+s} W_2(1)  \bigg|
\end{align*}
and it remains to show that the distribution on the right-hand side of the last display is identical to the distribution on the right-hand side of \eqrefb{eq:ThmMainH01}.
\begin{lemma}[Simplify limit distribution]\label{lem:simplylimit}
It holds that
\begin{align*}
&\sup_{0 \leq t < \infty} \max_{0 \leq s \leq t} w(t)
\bigg| W_1(t)- \dfrac{1+t}{1+s}W_1(s) + \dfrac{t-s}{1+s} W_2(1)  \bigg|\\
&\hspace{2.5cm}\eqd \sup_{0 \leq t < \infty} \max_{0 \leq s \leq t} (t+1)w(t)
\Big| W\Big(\dfrac{t}{t+1}\Big) -  W\Big(\dfrac{s}{s+1}\Big) \Big|~,
\end{align*}
where $W$ is a standard $p$-dimensional Brownian motion.
\end{lemma}
\begin{proof}[Proof of Lemma \ref{lem:simplylimit} and last step in the proof of Theorem \refb{thm:mainH0}]
$ $\newline
\renewcommand{\qedsymbol}{}
In the following let $Z$ denote a vector of $p$ independent standard Gaussian random variable, that is independent of $W_1$.
Observe that
\begin{align}\label{eq:lemmaBrownian}
&\;\;\;w(t) \bigg| W_1(t)- \dfrac{1+t}{1+s}W_1(s) + \dfrac{t-s}{1+s} W_2(1)  \bigg| \notag\\
&= \dfrac{w(t)}{(s+1)} \Big| (s+1)W_1(t) - (t+1)W_1(s) + (t-s)W_2(1) \Big|\notag\\
&\eqd \dfrac{w(t)}{(s+1)} \Big| (s+1)W(t) - (t+1)W(s) -(t-s)Z \Big|\notag\\
&= \dfrac{w(t)}{(s+1)} \Big| (s+1)W(t) - (t+1)W(s) -(t-s)Z +stZ -stZ \Big|\notag\\
&= \dfrac{w(t)}{(s+1)} \Big| (s+1)\big\{W(t) -tZ\big\} - (t+1)\big\{W(s)-sZ\big\} \Big|~.
\end{align}
Following \citesuppl{Horvath2004}, \citesuppl{Fremdt2015}, computing the covariance function implies the following identity (in distribution)
\begin{align*}
\Big\{W(t) - tZ \Big\}_{t \geq 0} 
\eqd \Big\{ (1+t)W\Big(\dfrac{t}{t+1}\Big)\Big\}_{t \geq 0}~.
\end{align*}
Applying this to \eqref{eq:lemmaBrownian} yields
\begin{align*}
&\dfrac{w(t)}{(s+1)} \Big| (t+1)\big\{W(s) -sZ\big\} - (s+1)\big\{W(t)-tZ\big\} \Big|\\
\eqd &\dfrac{w(t)}{(s+1)} \Big| (t+1)(s+1)W\Big(\dfrac{s}{s+1}\Big) - (t+1)(s+1)W\Big(\dfrac{t}{t+1}\Big) \Big|\\
= & w(t) \Big| (t+1)W\Big(\dfrac{s}{s+1}\Big) - (t+1)W\Big(\dfrac{t}{t+1}\Big) \Big|~.
\end{align*}
This completes the proof of Lemma \ref{lem:simplylimit} and also of Theorem \refb{thm:mainH0}.\vspace{-0.5cm}
\end{proof}
\end{proof}

\begin{proof}[\textbf{Proof of Corollary \refb{cor:simplify}:}]
We proceed according to the proof of Theorem 3.1 of \cite{Fremdt2015}.
Using the definition $w_\gamma(t)=\bigg[(1+t)\max\Big\{\Big(\dfrac{t}{1+t}\Big)^\gamma,\, \varepsilon \Big\}\bigg]^{-1}$, we obtain that
\begin{align*}
&\sup_{0 \leq t < \infty} \max_{0 \leq s \leq t} w_\gamma(t) \Big| (t+1)W\Big(\dfrac{s}{s+1}\Big) - (t+1)W\Big(\dfrac{t}{t+1}\Big) \Big|\\
&\qquad = \sup_{0 \leq t < \infty} \max_{0 \leq s \leq t} \dfrac{1}{\max\big\{ \Big(\dfrac{t}{1+t}\Big)^\gamma,\, \varepsilon \Big\}}\Big| W\Big(\dfrac{s}{s+1}\Big) - W\Big(\dfrac{t}{t+1}\Big) \Big|\\
&\qquad = \sup_{0 \leq t < 1} \max_{0 \leq s \leq t} \dfrac{1}{\max\{ t^\gamma, \varepsilon\}}\Big|W(s) - W(t)\Big|~,
\end{align*}
where we used that the mapping $x \mapsto x/(1+x)$ is bijective and increasing on the domain $[0,\infty)$ with co-domain $[0,1)$~.
\end{proof}

\begin{proof}[\textbf{Proof of Theorem \refb{thm:mainH1}:}]
For the ease of reading assume in the proof that $c_a m,\, c_ak^*_m \in \N$.
We follow the idea of \citesuppl{Stoehr2019} and distinct the cases $k_m^*/m = O(1)$ and $k_m^*/m \to \infty$.
In the first case, observe the lower bounds
\begin{align*}
\supkinf w(k/m)\hatE_m(k)
&\hspace{0.2cm}= \supkinf \dfrac{w(k/m)}{\sqrt{m}} \max_{j=0}^{k-1} (k - j) \Big\Vert \hattheta_{1}^{m+j} - \hattheta_{m+j+1}^{m+k} \Big\Vert_{\hatSigma^{-1}}\\
&\hspace{-1cm}\overset{k= c_am + k_m^*}{\geq}\dfrac{w\Big( \tfrac{k_m^*}{m} + c_a\Big)}{\sqrt{m}} 
\max_{j=0}^{c_am+ k_m^*-1}(c_am + k^*_m -j) \Big\Vert \hattheta_{1}^{m+j} - \hattheta_{m+j+1}^{(1+c_a)m+k_m^*} \Big\Vert_{\hatSigma^{-1}}~.
\end{align*}
Note that by Assumption \ref{assump:weighting} and \ref{assump:alt} it holds that $w\Big( \tfrac{k_m^*}{m} + c_a\Big) = \tildew\Big( \tfrac{k_m^*}{m} + c_a\Big)$ since $\tfrac{k_m^*}{m} + c_a \in (t_w,T_w)$.
Thereby, the last display equals
\begin{align*}
&\dfrac{\tildew\Big( \tfrac{k_m^*}{m} + c_a\Big)}{\sqrt{m}} \max_{j=0}^{c_am+ k_m^*-1}(c_am + k^*_m -j) \Big\Vert \hattheta_{1}^{m+j} - \hattheta_{m+j+1}^{(1+c_a)m+k_m^*} \Big\Vert_{\hatSigma^{-1}}\\
&\hspace{1.2cm}\overset{j=k^*_m-1}{\geq} \dfrac{(c_am+1)\tildew\Big( \tfrac{k_m^*}{m} + c_a\Big)}{\sqrt{m}}
\Big\Vert \hattheta_{1}^{m+k^*_m-1} - \hattheta_{m+k^*_m}^{(1+c_a)m+k_m^*} \Big\Vert_{\hatSigma^{-1}}~.
\end{align*}
Using the reverse triangle inequality, the last display is bounded from below by
\begin{align}\label{altlowerbound}
\begin{split}
&c_a \tildew\Big( \tfrac{k_m^*}{m} + c_a\Big)\\
&\hspace{1cm}\cdot
\sqrt{m} \Bigg(  \Big\Vert \theta^{(1)}_m - \theta^{(2)}_m  \Big\Vert_{\hatSigma^{-1}}
- \Big\Vert \hattheta_{1}^{m+k^*_m-1} - \theta^{(1)}_m  \Big\Vert_{\hatSigma^{-1}}
- \Big\Vert \hattheta_{m+k^*_m}^{(1+c_a)m+k_m^*} - \theta^{(2)}_m  \Big\Vert_{\hatSigma^{-1}} \Bigg)~.
\end{split}
\end{align}
To examine the first factor of the last display, note that we have $k^*_m/m \leq C$ for all $m \in \N$ and a sufficiently large constant $C$.
Using Assumption \refb{assump:weighting}, we now obtain
\begin{align*}
c_a \tildew\Big( \tfrac{k_m^*}{m} + c_a\Big)
\geq c_a\min\limits_{t \in [c_a,C+c_a]}\tildew(t)
> 0~.
\end{align*}
Now it remains to treat the second factor in \eqref{altlowerbound}.
Note that by Assumption \refb{assump:alt}, we obtain that
\begin{align*}
\sqrt{m} \Big| \theta^{(1)}_m - \theta^{(2)}_m  \Big| \convm \infty~.
\end{align*}
Using also the linearization in equation \eqrefb{eq:meta-remainder} and \eqrefb{alt:afterChange1}, we conclude that
\begin{align*}
\sqrt{m}\Big( \hattheta_{1}^{m+k^*_m-1} - \theta^{(1)}_m \Big)
= \dfrac{\sqrt{m}}{m+k^*_m-1}\sum_{t=1}^{m+k^*_m-1} \IF_t \;+\;\sqrt{m}R_{1,m+k^*_m-1}
= \Op(1)
\end{align*}
and
\begin{align*}
\sqrt{m} \Big( \hattheta_{m+k^*_m}^{(1+c_a)m+k^*_m} - \theta^{(2)}_m \Big)
= \dfrac{\sqrt{m}}{c_am+1}\sum_{t=m+k^*_m}^{(1+c_a)m+k^*_m} \IF_t \;+\; \sqrt{m} R_{m+k^*,(1+c_a)m+k^*_m}
= \Op(1)~.
\end{align*}
Putting all together and using also that $\hatSigma$ is (weakly) convergent with non-singular limit the treatment of the first case is finished since \eqref{altlowerbound} diverges to $\infty$.

It remains to treat the case $k_m^*/m \to \infty$, for which we can employ very similar arguments.
Setting $k = k_m^*(1+c_a) $ and $j= k_m^*-1$ in the definition of $w(k/m)\hatE_m(k)$ gives the lower bound
\begin{align}\label{altcase2}
\supkinf w(k/m)\hatE_m(k)
\geq \dfrac{c_ak_m^* w\Big( \tfrac{k_m^*}{m}(1 + c_a)\Big)}{\sqrt{m}}
\Big\Vert \hattheta_{1}^{m+k^*_m-1} - \hattheta_{m+k^*_m}^{m+(1+c_a)k_m^*} \Big\Vert_{\hatSigma^{-1}}~.
\end{align}
As $T_w=\infty$ by assumption, we have $w\Big( \tfrac{k_m^*}{m}(1 + c_a)\Big) = \tildew\Big( \tfrac{k_m^*}{m}(1 + c_a)\Big)$ for $m$ sufficiently large.
Now we obtain that \eqref{altcase2} has the lower bound
\begin{align}\label{altlowerbound2}
\begin{split}
&c_a \dfrac{k_m^*}{m} \tildew\Big( \tfrac{k_m^*}{m}(1 + c_a)\Big)\\
&\hspace{0.6cm}\cdot
\sqrt{m} \Bigg(  \Big\Vert \theta^{(1)}_m - \theta^{(2)}_m  \Big\Vert_{\hatSigma^{-1}}
- \Big\Vert \hattheta_{1}^{m+k^*_m-1} - \theta^{(1)}_m  \Big\Vert_{\hatSigma^{-1}}
- \Big\Vert \hattheta_{m+k^*_m}^{m+(1+c_a)k_m^*} - \theta^{(2)}_m  \Big\Vert_{\hatSigma^{-1}} \Bigg)~.
\end{split}
\end{align}
By assumption \eqrefb{lim:thresholdExtra} we obtain
$$
\liminf_{m \to \infty} \dfrac{k_m^*}{m} \tildew\Big( \tfrac{k_m^*}{m}(1 + c_a)\Big) > 0~.
$$
Using \eqrefb{alt:afterChange2} and repeating the corresponding steps from the first case, it follows that
$$
\sqrt{m} \Bigg(  \Big\Vert \theta^{(1)}_m - \theta^{(2)}_m  \Big\Vert_{\hatSigma^{-1}}
- \Big\Vert \hattheta_{1}^{m+k^*_m-1} - \theta^{(1)}_m  \Big\Vert_{\hatSigma^{-1}}
- \Big\Vert \hattheta_{m+k^*_m}^{m+(1+c_a)k_m^*} - \theta^{(2)}_m  \Big\Vert_{\hatSigma^{-1}} \Bigg)
\convp \infty~.
$$
Combining the last two statements with the lower bound provided in \eqref{altlowerbound2} the treatment of the second case and thereby the proof of Theorem \refb{thm:mainH1} is finished.
\end{proof}

\newpage
\section{Additional simulation results for Section \refb{sec4}}\label{sec:closedEnd}
In this section we provide some additional simulation results complementing the discussion on the power of the different monitoring procedures in Section \refb{sec4}.
The simulation settings are identical to those used in Section \refb{sec4} and as the results below are very similar to the results displayed in Figures \refb{fig:OE1} and \refb{fig:OE2} we omit a further discussion here.
\begin{figure}[H]
\includegraphics[width=15cm,height=8.85cm]{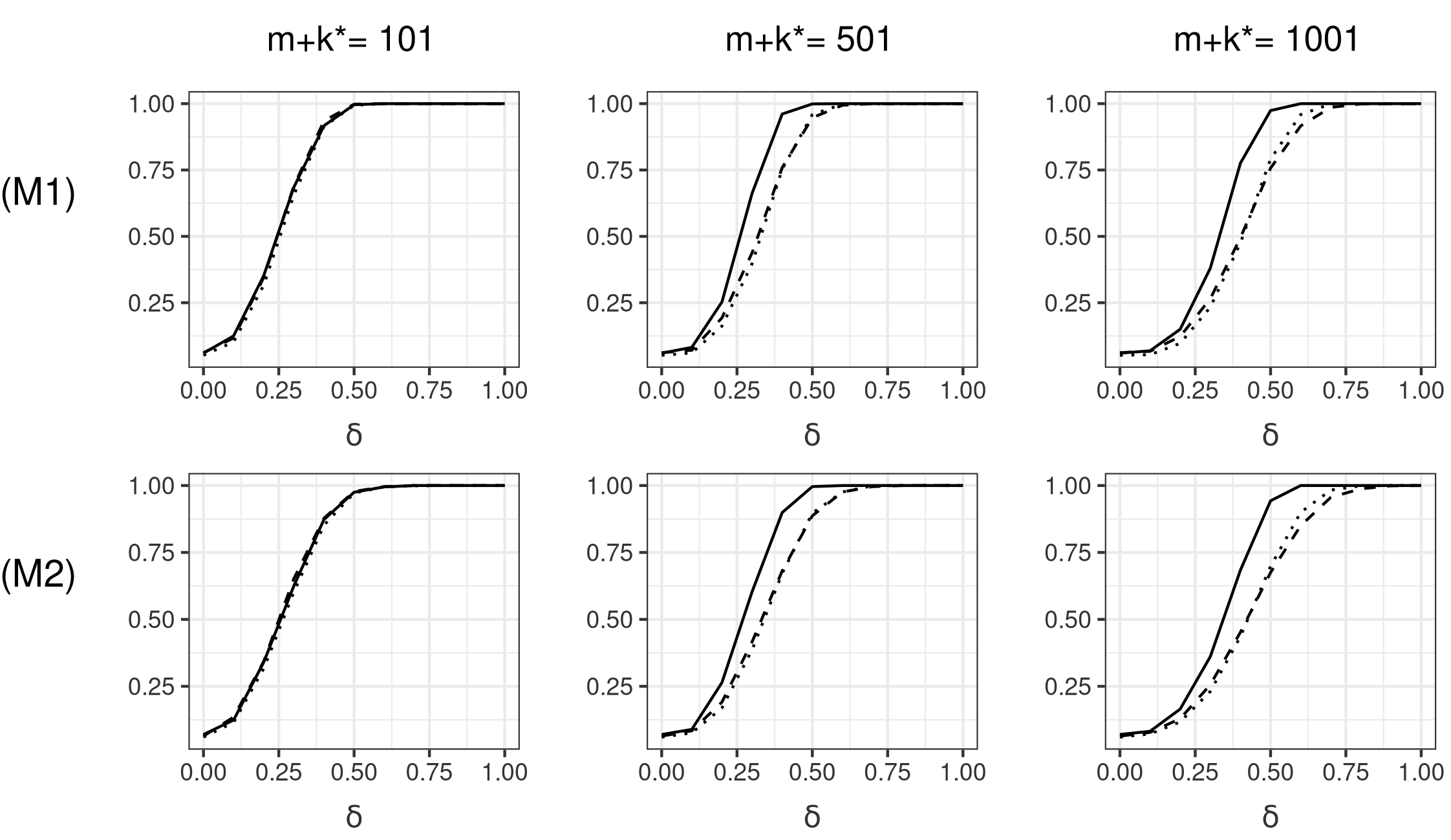} 
\caption{\it Power of the monitoring procedures for a change in the mean based on the statistics  $\hatE$ (solid line), $\hatQ$ (dashed line) and $\hatP$ (dotted line) with $\gamma=0.45$ and $m=100$.
\label{fig:OE3}}
\end{figure}
\begin{figure}[H]
\includegraphics[width=15cm,height=8.85cm]{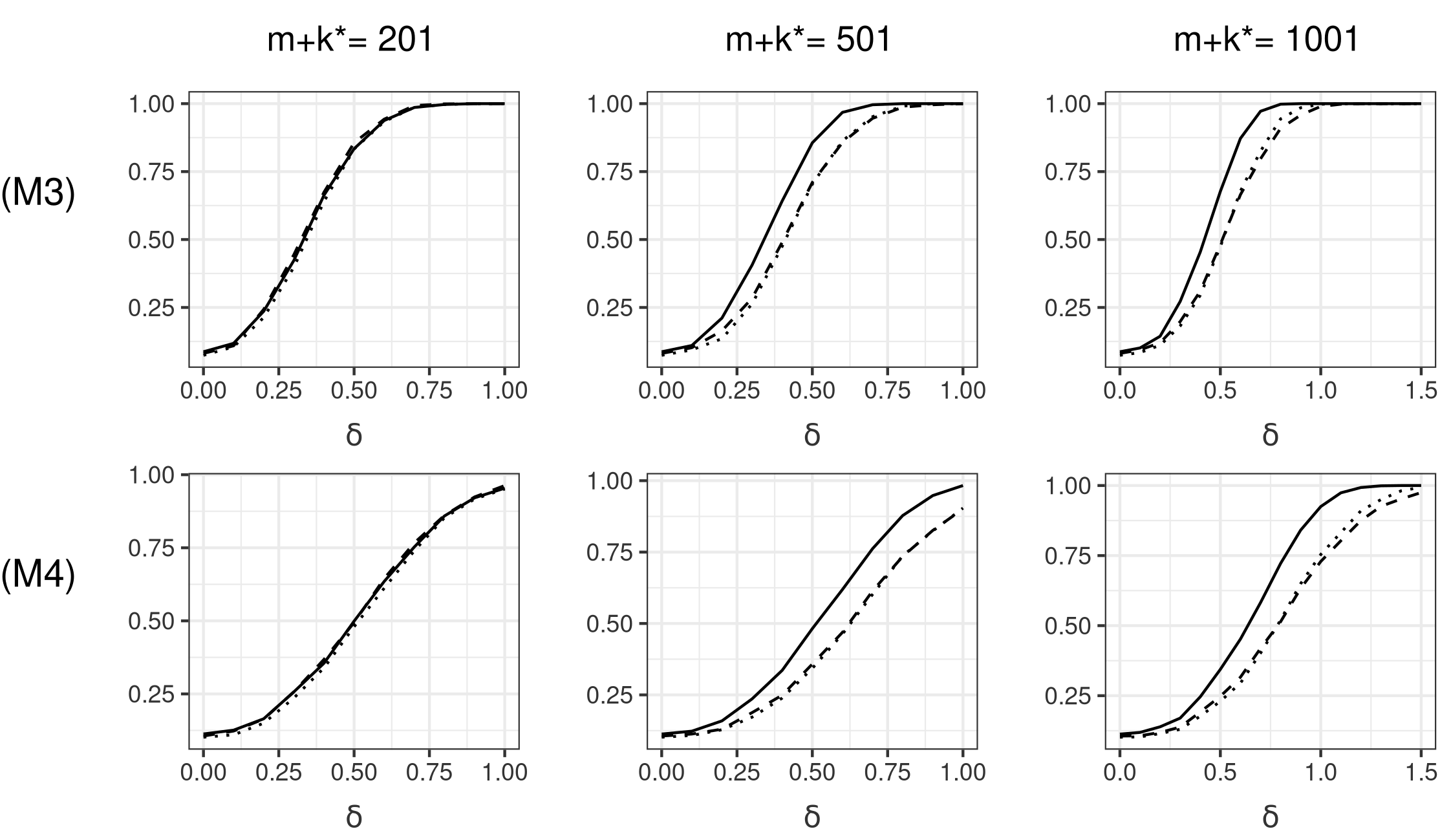} 
\caption{\it Power of the monitoring procedures for a change in the mean based on the statistics $\hatE$ (solid line), $\hatQ$ (dashed line) and $\hatP$ (dotted line) with $\gamma=0.45$ and $m=200$.
\label{fig:OE4}}
\end{figure}

\section{Closed-end scenarios}\label{sec:ClosedEnd}
It is worthwhile to mention that the theory developed in Section \refb{sec2} also covers the case of closed-end scenarios [sometimes also called finite time horizon].
In this section, we will very briefly discuss this situation and present a small batch of simulation results, which also indicate the superiority of the statistic $\hatE$ for closed-end scenarios.
Note that the null hypothesis in this setup is given by
\begin{align}\label{ClosedHypo0}
H_0&:\; \theta_1 = \dots = \theta_m = \theta_{m+1} =\theta_{m+2} = \ldots =\theta_{(T+1)m}~,
\end{align}
which is tested against the alternative that the parameters changes (once) at some time $m+1 \leq m+k^{\star} \leq (T+1)m$, that is
\begin{align}\label{ClosedHypo1}
 H_1&:\; \exists k^{\star} \in \N:\;\;
\theta_1 = \dots =\theta_{m+k^{\star}-1} \neq \theta_{m+k^{\star}} = \theta_{m+k^{\star}+1} = \ldots =\theta_{(T+1)m}~.
\end{align}
Here the factor $T \in \N$ controls the length of the monitoring period compared to the size of the initial data set.
Under the assumptions stated in Section 2, we can prove a corresponding statement of Theorem \refb{thm:mainH0} and Corollary \refb{cor:simplify}.

\begin{theorem}\label{thm:ClosedH0}
Assume that the null hypothesis \eqref{ClosedHypo0} and Assumptions \refb{assump:approx} - \refb{assump:remainder} hold.
If further $\hatSigma$ is a consistent and non-singular estimator of the long-run variance matrix $\Sigma$ it holds that
\begin{align}\label{eq:ThmClosedH0}
\sup_{k=1}^{Tm} w_\gamma(k/m) \hatE_m(k) 
&\convd
\sup_{0 \leq t \leq T} \max_{0 \leq s \leq t} (t+1)w(t)
\Big| W\Big(\dfrac{t}{t+1}\Big) -  W\Big(\dfrac{s}{s+1}\Big) \Big|
\end{align}
where $W$ is a $p$-dimensional Brownian motion with independent components.
Using $w=w_\gamma$ for the class of weight functions defined in \eqrefb{eq:threshold}, we further obtain that
\begin{align}\label{eq:closedAltRep}
\begin{split}
&\sup_{0 \leq t \leq T} \max_{0 \leq s \leq t} (t+1)w_\gamma(t)
\Big| W\Big(\dfrac{t}{t+1}\Big) -  W\Big(\dfrac{s}{s+1}\Big) \Big|\\
&\hspace{2cm}\eqd \max_{0 < t \leq T/(T+1)} \max_{0 \leq s \leq t}
\dfrac{1}{\max\{ t^\gamma, \varepsilon\}} \Big| W(t) - W(s) \Big|=:L_{1,\gamma}(T)~.
\end{split}
\end{align}

\end{theorem}
\noindent
The proof of Theorem \refb{thm:ClosedH0} follows from Theorem \refb{thm:mainH0} by using the factor $T$ as the cutoff $T_w$ of the weight function in \eqrefb{def:cutoffsweighting}.
The representation provided in \eqref{eq:closedAltRep} follows from a straightforward adaption of Corollary \refb{cor:simplify}.
The corresponding results for the tests based on statistics $\hatQ$ and $\hatP$ defined in \eqrefb{eq:otherStatistics} read as follows
\begin{align}\label{conv:limitQ2}
\max_{k=1}^{Tm} w_\gamma(k/m)\hat{Q}_m(k) \convd
\max_{0 < t < T/(T+1)} \dfrac{| W(t) |}{\max\{ t^\gamma, \varepsilon\}} =: L_{2,\gamma}(T)
\end{align}
and
\begin{align}\label{conv:limitP2}
\max_{k=1}^{Tm} w_\gamma(k/m)\hat{P}_m(k)  \convd \max_{0 < t < T/(T+1)} \max_{0 \leq s \leq t} \dfrac{1}{\max\{ t^\gamma, \varepsilon\}} \Big| W(t) - \dfrac{1-t}{1-s}W(s) \Big| =: L_{3,\gamma}(T) ~.
\end{align}
Likewise to Remark \refb{rem:exactcdf1} we can obtain an exact formula for the distribution of $L_{1,0}(T)$ in the case $p=1$ from page 146 of \citesuppl{Borodin1996}, this is
\begin{align}\label{eq:BorodinFormulaT}
F_{L_1(T),\gamma=0}(x)
= 1 + 8 \sum_{k=1}^\infty (-1)^k \cdot k  \cdot \Big(1- \Phi\big(kx/\sqrt{q(T)}\big)\Big)~,
\end{align}
where $\Phi$ denotes the c.d.f. of a standard Gaussian random variable and $q(T)$ denotes the quotient $T/(T+1)$.

To complete the discussion on closed-end scenarios we will display a small batch of simulation results for the detection of changes in the mean as described in Section \refb{sec:mean}.
For the sake of brevity, only the choice $T=4$ is examined here [unpublished simulation results show similar outcomes for other choices of $T$].
The remaining simulation settings are the same as used for the simulation study presented in Section \refb{sec:simMean} and in Table \ref{table:criticalp1CE} we display the necessary critical values defining the rejection regions for the different procedures. 

The approximation of the nominal level under the null hypothesis is displayed in Tables \ref{tab:meanCE1} and \ref{tab:meanCE2}
and in Figures \ref{fig:CE1} and \ref{fig:CE2} the power of the different procedures with respect to change amount and change position for $\gamma=0$ is illustrated.
The results are very similar to the open-end scenario discussed in Section \refb{sec4} and confirm the findings of that Section.

\begin{table}[H]
\begin{tabular}{c|ccc|ccc|ccc}
& \multicolumn{3}{c|}{$L_{1,\gamma}(4)$} & \multicolumn{3}{c|}{$L_{2,\gamma}(4)$} & \multicolumn{3}{c}{$L_{3,\gamma}(4)$} \\
\hline
 $\gamma$ \textbackslash $\alpha$ & 0.01 & 0.05 & 0.1 & 0.01 & 0.05 & 0.1 & 0.01 & 0.05 & 0.1 \\
\hline
\hline
0   & 2.7042 & 2.2339 & 2.0046 & 2.5145 & 1.9826 & 1.7380 & 2.5572 & 2.0435 & 1.8019\\
\hline
0.25& 2.9558 & 2.4345 & 2.2220 & 2.7602 & 2.2223 & 1.9799 & 2.8210 & 2.2986 & 2.0750\\
\hline
0.45& 3.3850 & 2.9371 & 2.6994 & 3.2238 & 2.7398 & 2.4952 & 3.3156 & 2.8626 & 2.6274\\
\end{tabular}
\caption{\it (1-$\alpha$)-quantiles of the distributions $L_{1,\gamma}(4)$, $L_{2,\gamma}(4)$ and $L_{3,\gamma}(4)$ for different choices of $\gamma$.
The cutoff constant was set to $\varepsilon=0$ and the dimension is $p=1$.
The quantiles for $L_{1,0}(4)$ were computed with respect to formula \eqref{eq:BorodinFormulaT}.
\label{table:criticalp1CE}}
\end{table}
\begin{table}[H]
\centering
\begin{tabular}{ccccccc}
 & \multicolumn{3}{c}{(M1)} & \multicolumn{3}{c}{(M2)}\\
\hline
\hline
&&&&&&\\[-1em]
$\gamma$ & $\hatE$ & $\hatQ$ & $\hatP$ &$\hatE$ & $\hatQ$ & $\hatP$ \\ 
\hline
\hline
$0$    & 5.0\% & 5.3\% & 5.3\% & 8.0\% & 7.3\% & 7.4\%\\
\hline
$0.25$ & 5.4\% & 6.0\% & 5.8\% & 8.6\% & 7.5\% & 8.1\%\\
\hline
$0.45$ & 4.9\% & 5.4\% & 4.5\% & 6.1\% & 6.4\% & 5.9\%\\
\hline
\end{tabular}
\caption{Type I error for the closed-end procedures for a change in the mean based on the statistics $\hatE$, $\hatQ$ and $\hatP$ at 5\% nominal size with a training data set of size $m=200$ and a monitoring window of $T=4$.
}
\label{tab:meanCE1}
\end{table}
\begin{table}[H]
\centering
\begin{tabular}{ccccccc}
 & \multicolumn{3}{c}{(M3)} & \multicolumn{3}{c}{(M4)}\\
\hline
\hline
&&&&&&\\[-1em]
$\gamma$ & $\hatE$ & $\hatQ$ & $\hatP$ &$\hatE$ & $\hatQ$ & $\hatP$ \\ 
\hline
\hline
$0$    & 7.6\% & 7.8\% & 8.1\% & 9.9\% & 9.8\% & 10.1\% \\
\hline
$0.25$ & 8.4\% & 8.2\% & 8.7\% & 11.4\% & 10.4\% & 10.8\% \\
\hline
$0.45$ & 6.6\% & 6.5\% & 6.8\% & 8.9\% & 8.3\% & 8.5\% \\
\hline
\end{tabular}
\caption{Type I error for the closed-end procedures for a change in the mean based on the statistics $\hatE$, $\hatQ$ and $\hatP$ at 5\% nominal size with a training data set of size $m=400$ and a monitoring window of $T=4$.
}
\label{tab:meanCE2}
\end{table}

\begin{figure}[H]
\includegraphics[width=15cm,height=8.85cm]{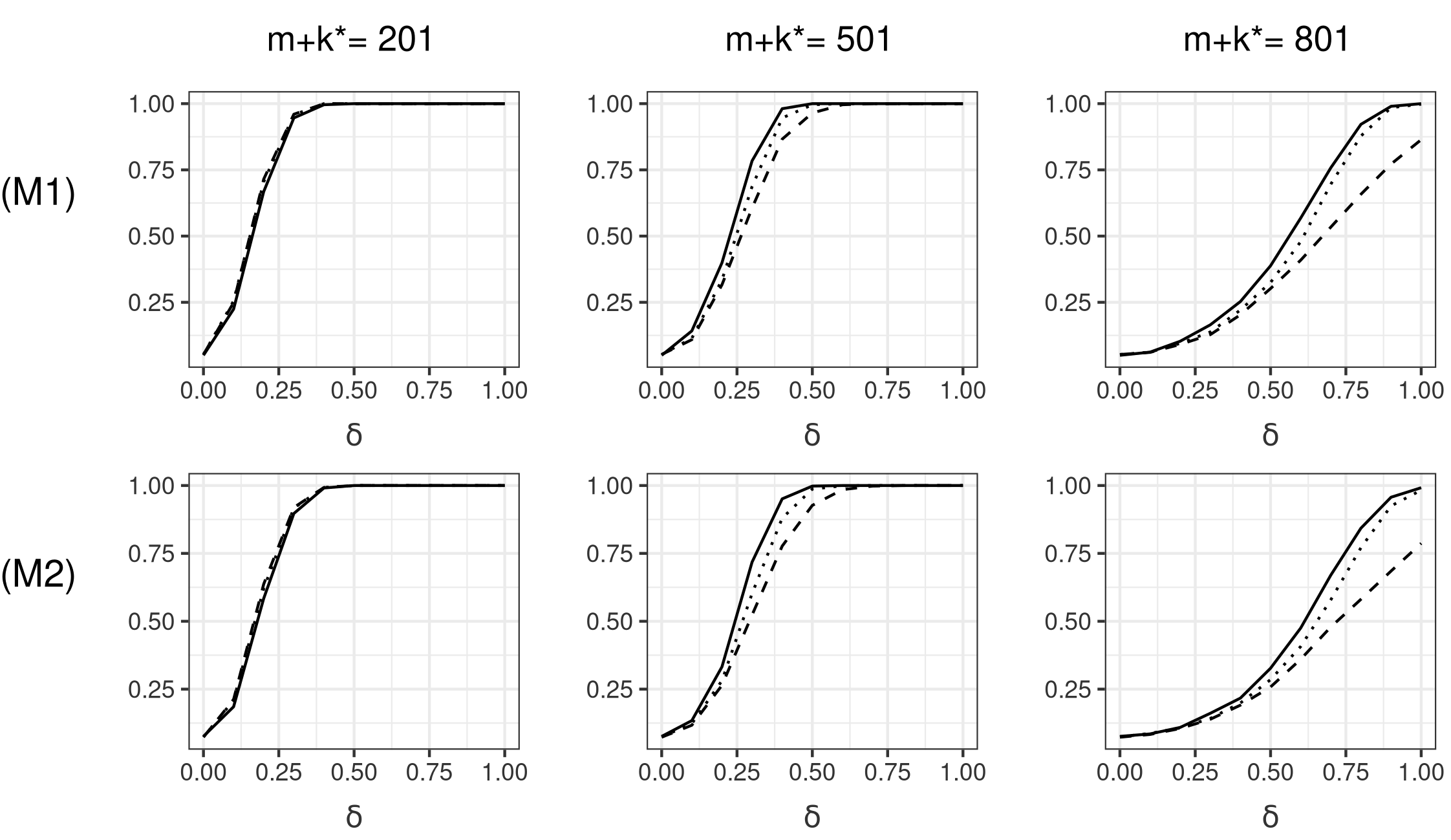} 
\vspace{-0.4cm}
\caption{\it Power of the (closed-end) monitoring procedures for a change in the mean based on the statistics $\hatE$ (solid line), $\hatQ$ (dashed line) and $\hatP$ (dotted line) with $\gamma=0$, $m=200$ and $T=4$.
\label{fig:CE1}}
\end{figure}
\begin{figure}[H]
\includegraphics[width=15cm,height=8.85cm]{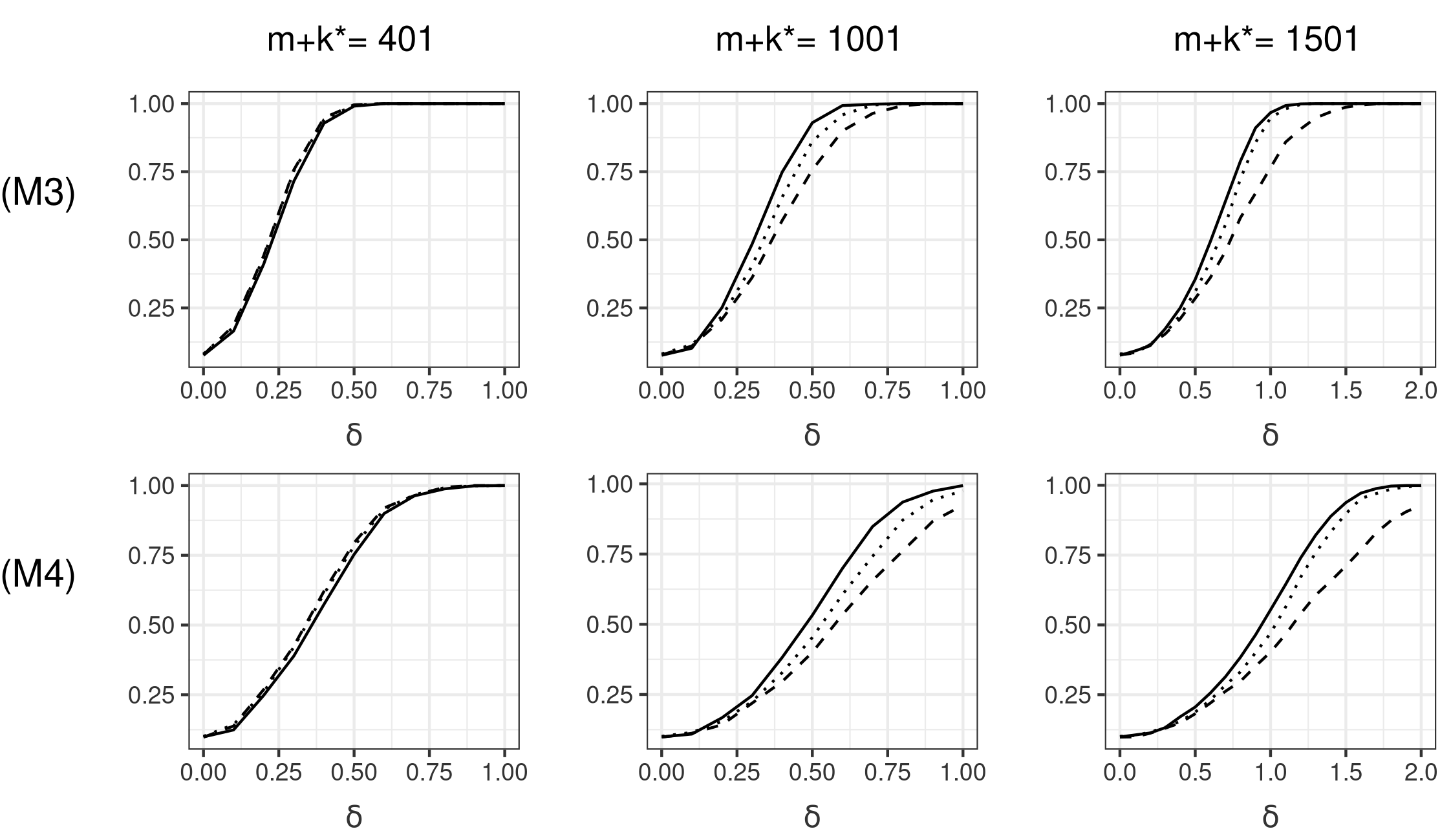} 
\vspace{-0.4cm}
\caption{\it Power of the (closed-end) monitoring procedures for a change in the mean based on the statistics $\hatE$ (solid line), $\hatQ$ (dashed line) and $\hatP$ (dotted line) with $\gamma=0$, $m=400$ and $T=4$.
\label{fig:CE2}}
\end{figure}

\newpage
\bibliographystylesuppl{apalike}
\bibliographysuppl{literature}

\end{document}